\documentclass[12pt, reqno]{amsart}
\usepackage{amsfonts,amssymb,latexsym,amsmath, amsxtra, nccmath, enumerate, mathtools}
\usepackage[all]{xy}
\usepackage[margin=1in]{geometry}
\usepackage[OT2,T1]{fontenc}
\usepackage{times}
\usepackage[T1]{fontenc}
\usepackage{mathrsfs}
\usepackage{color}

\DeclareSymbolFont{cyrletters}{OT2}{wncyr}{m}{n}

\theoremstyle{plain}
\newtheorem{theorem}{Theorem}[section]
\newtheorem{corollary}[theorem]{Corollary}

\newtheorem{lemma}[theorem]{Lemma}
\newtheorem{proposition}[theorem]{Proposition}

\theoremstyle{definition}
\newtheorem{definition}[theorem]{Definition}
\newtheorem{rem}[theorem]{Remark}

\theoremstyle{remark}
\newtheorem*{remark}{Remark}

\newtheorem*{question}{Question}
\numberwithin{equation}{section}

\theoremstyle{plain}
\newcommand{\thistheoremname}{}
\newtheorem*{genericthm*}{\thistheoremname}
\newenvironment{namedthm*}[1]
  {\renewcommand{\thistheoremname}{#1}%
   \begin{genericthm*}}
  {\end{genericthm*}}

\newcommand{\ZZ}{{\mathbb Z}}
\newcommand{\RR}{{\mathbb R}}
\newcommand{\HH}{{\mathbb H}}

\newcommand{\CC}{{\mathbb C}}
\newcommand{\sM}{{\mathcal M}}

\DeclareMathOperator{\SL}{SL}
\DeclareMathOperator{\PSL}{PSL}

\newcommand{\Z}{\mathbb Z}
\newcommand{\C}{\mathbb C}

\renewcommand{\H}{\mathbb H}

\newcommand{\smatr}[4]{ \left( \begin{smallmatrix} #1 & #2 \\ #3 & #4\end{smallmatrix} \right) }

\newcommand{\ol}[1]{\overline{{#1}}}

\newcommand{\re}[1]{\text{Re} \left( #1 \right) }
\newcommand{\im}[1]{\text{Im} \left( #1 \right) }
\newcommand{\abs}[1]{\left|#1\right|}

\newcommand{\htE}{{\widehat{E}}}

\newcommand{\dhtE}{\skew{2.9}\widehat{\vphantom{\rule{1pt}{10.5pt}}\smash{\widehat{E}}}}
\newcommand{\htz}{{\widehat{\zeta}}}

\newcommand{\sgn}[1]{{\rm sgn}(#1)}
\newcommand{\epp}{{\epsilon}}

\begin{document}

\title[Shifted Polyharmonic Maass Forms]{Shifted Polyharmonic Maass Forms for $\PSL(2,\ZZ)$}

\author{Nickolas Andersen}
\address{UCLA Dept. of Mathematics, Los Angeles, CA 90095}
\email{nandersen@math.ucla.edu}

\author{Jeffrey C. Lagarias}
\address{Dept. of Mathematics, The University of Michigan, Ann Arbor, MI 48109-1043, USA}
\email{lagarias@umich.edu}

\author{Robert C. Rhoades}
\address{Center for Communications Research, Princeton, NJ 08540}
\email{rob.rhoades@gmail.com}

\thanks{
Research of the first author was partially supported by NSF grant DMS-1440140, while he was in residence at MSRI in Berkeley, CA during the Spring 2017 semester, and also by NSF grant DMS-1701638.
Research of the second author was partially supported by
NSF grant DMS-1401224 and DMS-1701576.
Research of the third author was partially supported by an
NSF Mathematical Sciences Postdoctoral Fellowship.}

\date{\today}


\begin{abstract}
We study the vector space $V_k^m(\lambda)$ of shifted polyharmonic Maass forms of weight $k\in 2\Z$, depth $m\geq 0$, and shift $\lambda\in \C$.
This space is composed of real-analytic modular forms of weight $k$ for $\PSL(2,\Z)$ with moderate growth at the cusp which are annihilated by $(\Delta_k - \lambda)^m$, where $\Delta_k$ is the weight $k$ hyperbolic Laplacian.
We treat the case $\lambda \neq 0$, complementing work of the second and third authors on polyharmonic Maass forms (with no shift).
We show that $V_k^m(\lambda)$ is finite-dimensional and bound its dimension.
We explain the role of the real-analytic Eisenstein series $E_k(z,s)$ with $\lambda=s(s+k-1)$ and of the differential operator $\frac{\partial}{\partial s}$ in this theory.
\end{abstract}

\maketitle

%%%%%%%%%%%%%%%%%%%%%%%%%%%%%%%%%%%%%%%%%%%%%%
%%                                          %%
%%   1  INTRODUCTION                        %%
%%                                          %%
%%%%%%%%%%%%%%%%%%%%%%%%%%%%%%%%%%%%%%%%%%%%%%

\section{Introduction}\label{sec:Intro}
The second and third authors initiated a study of polyharmonic Maass forms for $\PSL(2, \ZZ)$ in \cite{LR15K}. The present paper
extends this study to a more general class of such forms. 
Fix $k\in 2\Z$ and let
\begin{equation*}
	\Delta_k := y^2  \left(  \frac{\partial^2}{\partial x^2} + \frac{\partial^2}{\partial y^2} \right)  
- i ky \left( \frac{\partial}{\partial x} + i \frac{\partial}{\partial y} \right) , \qquad z=x+iy \in \C,
\end{equation*}
denote the weight $k$ hyperbolic Laplacian.%
\footnote{We follow the convention of Maass \cite{Maa83}
for the sign of the Laplacian. Other authors call $-\Delta_k$
the hyperbolic Laplacian, e.g. \cite{DIT13}.}
A \emph{classical Maass form}%
\footnote{For $\lambda =0$, this definition includes classical holomorphic modular forms of weight $k$.
Other authors use a different definition of Maass forms. See Section~\ref{sec:Background} for details.}
of weight $k$ and eigenvalue $\lambda$ for $\PSL(2,\Z)$ is a smooth function $f:\H\to \C$ with moderate (at worst polynomial) growth at $i\infty$ which satisfies
\begin{equation*}
	f\left(\frac{az+b}{cz+d}\right) = (cz+d)^k f(z) \quad \text{ for all } \smatr abcd \in \PSL(2,\Z)
\end{equation*}
and for which
\begin{equation} \label{eq:laplace-eig-condition}
	(\Delta_k - \lambda) f = 0.
\end{equation}
A theory for more general Fuchsian groups $\Gamma \subseteq \SL(2, \ZZ)$ which
do not include $-I$ would allow odd integer weights as well, but no new forms are  gained here 
since all odd integer weight modular forms on $\PSL(2, \ZZ)$ must vanish identically.

In this work we study the situation when the Laplacian eigenfunction condition \eqref{eq:laplace-eig-condition} is relaxed to require only that
\begin{equation}\label{poly-h-eqn}
 \left( \Delta_k - \lambda \right) ^m f (z)= 0
\end{equation}
for some non-negative integer $m$. 
We denote the vector space of such functions by $V_k^m(\lambda)$,
and we call such $f$  {\em shifted polyharmonic Maass forms of depth $m$}  with {\em eigenvalue} $\lambda$ (or \emph{shifted m-harmonic Maass forms with eigenvalue $\lambda$}).
The integer parameter $m$  in \eqref{poly-h-eqn} is termed 
the  {\em (shifted) harmonic depth}  rather than  {\em  order} (as in PDEs)
because  the term  \emph{order} is 
used  in   conflicting ways in the literature (see \cite{LR15K}).

Our object in this paper is to establish properties of  the vector spaces $V_k^m(\lambda)$ for eigenvalue shifts $\lambda \ne 0$.
The case $\lambda=0$ was previously  treated   in \cite{LR15K}.
We  show that $V_k^m(\lambda)$ is finite-dimensional, determine an upper bound for the dimension,
and exhibit linearly independent forms in these spaces. 
The finite dimensionality of the  spaces $V_k^m(\lambda)$ has entirely to do with the moderate growth condition; if this is relaxed, then 
the resulting space of solutions can be infinite dimensional. As in \cite{LR15K}   
new members of the vector spaces over the harmonic depth $1$ case
involve  derivatives in the $s$-variable of (properly scaled) non-holomorphic Eisenstein series 
$E_k(z, s)$.

This case $\lambda=0$ has  exceptional properties
 which justify its separate treatment in \cite{LR15K}.
It includes all the holomorphic modular forms; the paper \cite{LR15K} explains that holomorphic 
forms should be assigned harmonic depth a half-integer.  
In that paper,  functions  in $V_k^m(0)$ are named   {\em polyharmonic Maass forms}, in parallel with the literature
on polyharmonic functions,
which are functions annihilated by
a power $\Delta^m$ of the Euclidean Laplacian.
(For work on Euclidean polyharmonic functions
see  Almansi \cite{Al1899}, Aronszajn, Crease and Lipkin \cite{ACL83}, Render \cite{Ren08}, Mitrea \cite[Chap. 7]{Mit13}.)

Our results involve the  non-holomorphic Eisenstein series $E_k(z, s)$ of weight $k$  for  $\PSL(2, \ZZ)$ given by
the series
\begin{equation}\label{Eis-k}
E_{k}(z, s) :=\frac{1}{2} \sum_{(m,n) \in \ZZ^2 \backslash  (0,0)} \frac{y^s}{|mz+n|^{2s} (mz+n)^{k}},
\end{equation}
which converges absolutely for $\re s > 1-\frac{k}{2}$, and has a meromorphic continuation in
the $s$-variable (see Section \ref{sec:NHE}). 
The series is well-defined for all $k\in \Z$, but  vanishes identically for odd $k$.
It is shifted $1$-harmonic with eigenvalue $\lambda= s(s+k-1)$.
We will consider the {\em doubly-completed Eisenstein series}
$$
\dhtE_k(s) := \left(s+ \mfrac{k}{2}\right) \left(s +\mfrac{k}{2}-1\right)\pi^{-s-\frac{k}{2}}\Gamma\left(s+ \mfrac{k}{2} + \mfrac{|k|}{2}\right) E_k(z, s),
$$
which one can show is an entire function of $s$ for each $z \in \HH$ (see Section~\ref{sec:NHE}).
We write the Taylor series expansions of the doubly-completed Eisenstein series 
at $s=s_0 \in \CC$ as
$$
\dhtE_{k}(z, s) = \sum_{j=0}^{\infty} \frac{1}{j!} \dhtE_k^{[j]}(z; s_0) (s- s_0)^j.
$$

%%%%%%%%%%%%%%%%
% Theorem 1.1  %
%%%%%%%%%%%%%%%%

\begin{theorem}\label{thm:main1}
Fix $k\in 2\Z$. 
For $\lambda\in \CC$ fix $s_0 \in \CC$ such that $\lambda= s_0( s_0 +k -1)$.
\begin{enumerate}
\item The complex vector space $V_k^m(\lambda)$ is finite dimensional, with
\begin{equation*}
	\dim V_k^m(\lambda) \leq m + m \dim S_k^1(\lambda),
\end{equation*}
where $S_k^1(\lambda)$ is the space of Maass cusp forms of weight $k$ and eigenvalue $\lambda$.

\item This space decomposes as
\[
	V_k^m(\lambda) =E_k^m(\lambda) \oplus S_k^m(\lambda),
\]
in which the Eisenstein series space $E_k^m(\lambda)$ is spanned by certain 
Taylor coefficients of shifted Eisenstein series  and $S_k^m(\lambda)$ is
a recursively defined space of  ``generalized $m$-harmonic Maass cusp forms.''
Both vector spaces $E_k^m(\lambda)$ and
$S_k^m(\lambda)$ are closed under the action of $\Delta_k - \lambda$.
	
\item For all  $\lambda \in \CC$ the space $E_k^m(\lambda)$ has dimension $m$.
\begin{enumerate}[(i)]
 	\item
    For $\lambda \ne -(\frac{1-k}{2})^2$  it  has a basis consisting of 
    the Taylor coefficient functions 
    \[
    	\dhtE^{[j+r]}(z; s_0) := \frac{\partial^{j+r}}{\partial s^{j+r}} \dhtE_k(z; s) \big|_{s= s_0}
		\quad \mbox{for} \quad 0 \le j \le m-1,
	\]
 	where $r$ is minimal such that 
	$\dhtE_k^{[r]}(z; s_0) \not\equiv 0$. Here $r=0$ unless  $\lambda=\frac k2(1-\frac k2)$
	and $k \ne 0$, in which case   $r=1$. 
	
	\item
	For $\lambda = - (\frac{1-k}{2})^2$ and $s_0 =\frac{1-k}{2}$,  a basis is given by  the even-indexed Taylor
	coeffcient functions $\dhtE^{[2j]}(z; s_0)$ 
	for $0 \le j \le m-1$. All odd-indexed functions $\dhtE^{[2j+1]}(z; s_0) \equiv 0$.
\end{enumerate}

\item 
For $m \ge 1$ one has 
\[
	\dim \big( S_k^{m}(\lambda)\big) \le m \dim \big(S_k^1(\lambda) \big).
\]
\end{enumerate}
\end{theorem}

While the space $S_k^{1}(\lambda)$
comprises the usual Maass cusp forms,
Theorem \ref{thm:main1} leaves open the question of determining 
the dimension of the space of ``generalized $m$-harmonic Maass cusp forms''  $S_k^m(\lambda)$ for $m \ge 2$.
In  \cite[Sect. 6.3]{LR15K} the second and third authors showed for $\lambda=0$ (the holomorphic cusp form case) that  $S_k^2(0) = S_k^1(0),$
which in turn forces $S_k^m(0) = S_k^1(0)$ for all $m \ge 1$. The  proof given there  was
 specific to the assumption $\lambda=0$.
One may ask:
\begin{question}
	Is it true for general $\lambda \in \CC$  that $S_k^m(\lambda) = S_k^1(\lambda)$ for all $m \ge 1$?
\end{question}

This question asks whether there is a nonzero  element of $S_k^1$ that has a ``liftability'' property, i.e. is the image of some element in
$V_k^2(\lambda)$.
If some cusp forms can be ``lifted'', then we warn the reader that the  term ``cusp form'' could be a misnomer in our recursive definition of $S_k^m(\lambda)$;
that is,  we do not know whether all the forms in $S_k^m(\lambda)$ 
will have  identically zero  constant term in their  Fourier expansion. 

The next two results concern the preservation of the spaces $V_k^m(\lambda)$ under various differential operators.
The first of these  concerns the actions of the {\em weight $k$ Maass  raising operator}
\[
R_k := 2i \frac{\partial}{\partial z} + \frac{k}{y} = i\left( \frac{\partial}{\partial x} - i \frac{\partial}{\partial y} \right) + \frac{k}{y},
\]
and {\em weight $k$ Maass lowering operator}
\[
L_k := 2i y^2\frac{\partial}{\partial \bar{z}}  = iy^2\left( \frac{\partial}{\partial x} + i \frac{\partial}{\partial y} \right)
\]
on these vector spaces.  Note that the operator $L_k$ does not depend on $k$.

%%%%%%%%%%%%%%%%
% Theorem 1.2
%%%%%%%%%%%%%%%%
\begin{theorem}\label{thm:main3}
\begin{enumerate}
\item There holds
\begin{equation} \label{eq:thm-rk}
	R_k \big(V_k^m(\lambda)\big) \subset V_{k+2}^m(\lambda +k)
\end{equation}
and
\begin{equation} \label{eq:thm-lk}
	L_k\big( V_k^m(\lambda)\big) \subset V_{k-2}^m(\lambda + 2-k).
\end{equation}

\item The map \eqref{eq:thm-rk} is an isomorphism when $\lambda+k\neq 0$.
The map \eqref{eq:thm-lk} is an isomorphism when $\lambda \neq 0$.
\end{enumerate}
\end{theorem}

This result is proved as Propositions~\ref{RLemma4} and \ref{prop:RkLk-iso}.
The key property established is that these operators 
acting on individual Fourier coefficients preserve
the moderate growth property.

Our second result for differential operators determines the action of 
the Bruinier-Funke differential  operator $\xi_k :=2i y^k \ol{\frac{\partial}{\partial \ol{z}}}$
on shifted polyharmonic forms
of weight $k$.

%%%%%%%%%%%%%%%%
% Theorem 1.3 
%%%%%%%%%%%%%%%%
\begin{theorem}\label{thm:main2}
For $k \in 2\ZZ$ and all $m \ge 1$, the  antiholomorphic differential operator $\xi_k$
maps $V_k^m(\lambda)$ to $V_{2-k}^m(\ol{\lambda})$.
For $\lambda \ne 0$ this map is a vector space isomorphism. 
\end{theorem}

This result is proved as Proposition~\ref{lem:subspaceXi}. 
Again a key property  established is that the operator $\xi_k$ 
preserves the  moderate growth property  of these functions at the cusp.

Figure \ref{fig:tower-ladder}  pictures the action of these maps, along with $\Delta_k - \lambda$,
acting on these spaces. It has a 
``tower'' and ``ladder'' structure for
weights $k \ge 2$, paired with the dual weight $2-k \le 0$,
at the level of vector spaces.
 
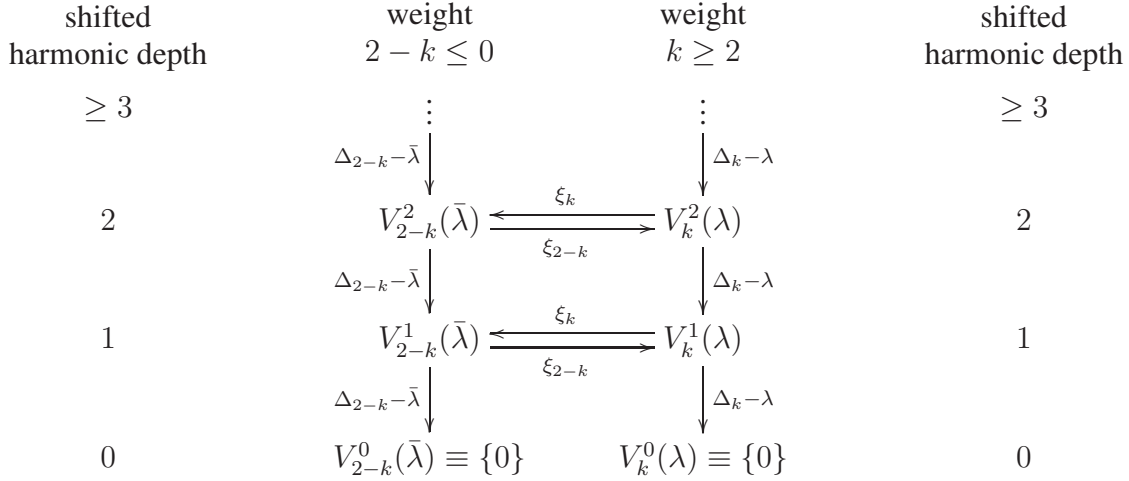
\begin{figure}[h]
\[
	\xymatrix{
		\parbox{1.5in}{\center shifted \\ harmonic depth \\[0.5em] $\geq 3$} &
		\parbox{1in}{\center weight \\ $2-k\leq 0$ \\[0.5em] $\vdots$} \ar[d]_(.6){\Delta_{2-k}-\bar\lambda}
				& \parbox{1in}{\center weight \\ $k\geq 2$ \\[0.5em] $\vdots$} \ar[d]^(.6){\Delta_{k}-\lambda}
				& \parbox{1.5in}{\center shifted \\ harmonic depth \\[0.5em] $\geq 3$} \\
		2 & V_{2-k}^2(\bar\lambda) \ar@<-0.5ex>[r]_{\xi_{2-k}} \ar[d]_{\Delta_{2-k}-\bar\lambda}
  			 & V_{k}^2(\lambda) \ar@<-0.5ex>[l]_{\xi_k} \ar[d]^{\Delta_{k}-\lambda} & 2   \\
  		1 & V_{2-k}^1(\bar\lambda) \ar@<-0.5ex>[r]_{\xi_{2-k}} \ar[d]_{\Delta_{2-k}-\bar\lambda}
  			& V_{k}^1(\lambda) \ar@<-0.5ex>[l]_{\xi_k} \ar[d]^{\Delta_{k}-\lambda} & 1  \\
  		0 & V_{2-k}^0(\bar\lambda) \equiv \{0\}  
  			& V_{k}^0(\lambda) \equiv \{0\}  & 0             
  	}
\]
\caption{Tower and ladder structure for shifted polyharmonic vector spaces (eigenvalue $\lambda\neq 0$).}
\label{fig:tower-ladder}
\end{figure}

When $\lambda=0$ the action of $\xi_k: V_k^m(0) \to V_{2-k}^m(0)$  is never an isomorphism.
In that case  there is, in contrast, a ``tower" and ``ramp" structure
between weights $k$ and $2-k$  for the action of these maps, cf.  \cite[Figures 1 and 2]{LR15K}.

The finite-dimensionality  results in Theorem \ref{thm:main1} are proved using general results on the form of 
Fourier expansions of  polyharmonic
functions $f$ having eigenvalue shift $\lambda$ that are invariant under  $z \mapsto z+1$.
The relevant property of Fourier coefficients of harmonic depth $m$  is that 
the  $2m$-dimensional space of harmonic depth $m$ eigenfunctions for an individual (non-constant term) Fourier coefficient 
contains an $m$-dimensional subspace comprising  the full set of eigenfunctions
having moderate growth (in fact, fast decay) at the cusp.  Our proof of this fact, given in Appendix A, rests on explicit calculation of
asymptotics of the eigenfunction families of the associated ordinary differential operators.
The proof of Theorem \ref{thm:main1} does not make use of the Bruinier-Funke anti-holomorphic differential operator $\xi_k$,
in contrast to \cite{LR15K}.

These results on the
allowed form of polyharmonic Fourier coefficients  apply  more generally to such expansions for subgroups of
the modular group at  cusps of any width, with rescaling.  They also
will apply to finite covering groups of $\SL(2, \RR)$, such as the metaplectic group $\operatorname{Mp}(2, \RR)$,
thus extending to  the case of half-integer
weight modular forms on suitable discrete subgroups. That is, one may expect derivatives in
the $s$-variable of (non-holomorphic) Eisenstein series to give 
polyharmonic modular forms more generally.

Section~\ref{sec:Background} 
contains known results on Maass cusp forms (Section~\ref{sec:MaassCusp}),
results concerned with Maass's calculations related to non-holomorphic Eisenstein series (Section~\ref{sec:MaassSeries}), and basic facts on non-holomorphic Eisenstein series,
including their Fourier expansions and functional equations (Section~\ref{sec:NHE}).
Section~\ref{sec:newsec4} discusses the Fourier expansions of polyharmonic
Maass forms. 
Section~\ref{sec:MaassOps}
studies the action of Maass raising and lowering operators on the vector spaces $V_k^m(\lambda)$.
Section~\ref{sec:newsec5}  presents results about the 
Bruinier-Funke 
non-holomorphic differential operator $\xi_k$,
showing that it preserves the property of moderate
growth for shifted polyharmonic Maass forms for any eigenvalue $\lambda$.
In Section~\ref{sec:Xi-action2} we compute the action of $\xi_k$ on the Eisenstein series.
In Section~\ref{sec:TS} we give recursions for the Taylor series coefficients
of $\dhtE_k(z; s_0)$ in the $s$-variable, which are functions of $z$,
and determine recursion relations for the action of $\Delta_k$ on
the Taylor series  coefficients. 
In Section~\ref{sec:Proofs} we give the proof  of 
Theorem \ref{thm:main1}.
There are two appendices treating subsidiary topics.
Appendix \ref{sec:AppA} gives asymptotic expansions of derivatives of
the Whittaker $W$-functions $W_{\kappa,\mu}(y)$ in the second index $\mu$.
Appendix \ref{sec:AppC} sketches  an alternate proof of Proposition~\ref{pr56} 
giving the action of the $\xi_k$-operator on 
the non-holomorphic Eisenstein series.

%%%%%%%%%%%%%%%%%%%%%%%%%%%%%%%%%%%%%%%%%%%%%
%%                                         %%
%%     2   BACKGROUND                      %%
%%                                         %%
%%%%%%%%%%%%%%%%%%%%%%%%%%%%%%%%%%%%%%%%%%%%%

\section{Background Results} \label{sec:Background}

The definition of Maass forms given in Section~\ref{sec:Intro} includes
classical holomorphic modular forms of weight $k \in 2\ZZ$ when $\lambda=0$;  
this definition was adopted for much recent work treating mock theta functions
and other mock modular forms.
We note however that the  original treatment of  Maass \cite{Maa53}, \cite[Chap. IV]{Maa83} and  
other later treatments  (e.g. Bump \cite[Sect. 2.1]{Bu97}, Duke, Friedlander, Iwaniec \cite[Sect. 4]{DFI02}) use a different 
definition of Maass form.
They define a Maass form of weight $k$ as a function $f^{[M]}$ with at worst polynomial growth at $i\infty$ which satisfies the modified modular invariance condition
\[
	f^{[M]} \left( \frac{az+b}{cz+d} \right)   = \left(\frac{ cz+d}{|cz+d|}\right)^{k} f^{[M]}(z) \quad \text{ for all }\smatr abcd \in \PSL(2,\Z)
\]
and is annihilated by $\Delta_k^{[M]}-\lambda$, where
\[
	\Delta_k^{[M]} := y^2  \left(  \frac{\partial^2}{\partial x^2} + \frac{\partial^2}{\partial y^2} \right)  
- i ky\frac{\partial}{\partial x}.
\]
The two definitions coincide for weight $k=0$ but differ otherwise. 
For $k \ne 0$ Maass forms in this sense are never holomorphic functions of $z$.
Given a Maass form $f$ of weight $k$ in our sense,
one can  transfer results to the definition above
by setting $f^{[M]}(z) := y^{\frac{k}{2}} f(z)$.

We begin by discussing Maass cusp forms of weight $0$ in Section~\ref{sec:MaassCusp}.
We then recall properties of Maass's Eisenstein series as discussed in \cite{Maa83} in Section~\ref{sec:MaassSeries}, and in Section~\ref{sec:NHE} we derive parallel properties for the
non-holomorphic Eisenstein series $E_k(z, s)$.

%%%%%%%%%%%%%%%%%%%%%%%%%%%%%%%%%%%%%%%%%%%%%
%%                                         %%
%%     2.1   Maass cusp forms       	   %%
%%                                         %%
%%%%%%%%%%%%%%%%%%%%%%%%%%%%%%%%%%%%%%%%%%%%%

\subsection{Maass cusp forms}\label{sec:MaassCusp}

A Maass cusp form is a Maass form with rapid decay at the cusp.
In 1949 Maass \cite{Maa49} introduced the weight $0$ case
for certain congruence subgroups of $\SL(2,\ZZ)$, so we begin there.

Suppose that $F$ is a Maass cusp form of weight $0$ with eigenvalue $\lambda$.
Writing $\lambda=s_0(s_0-1)$, such an $F$ has a Fourier series expansion of the form
\[
	F(z) = \sum_{n \in \ZZ \smallsetminus \{ 0\} } a_n \sqrt{y} K_{s_0 - \frac{1}{2}} (2 \pi |n|y) e^{2 \pi i n x},
\]
in which $K_{\nu}(z)$ is a $K$-Bessel function.
This expansion can also be given in terms of Whittaker functions $W_{\kappa,\mu}(y)$ via the relation
\[
	2 \sqrt{|n|y} K_{s-\frac 12}(2\pi |n|y) = W_{0,s-\frac 12}(4\pi |n| y)	
\]
(see Section~\ref{sec:newsec4}).
We let $S_0^{1}(\lambda)$ denote the space of Maass cusp forms of eigenvalue $\lambda$,
consistent with the notation of Theorem~\ref{thm:main1}.

%%%%%%%%%%%%
% Theorem 2.2 
%%%%%%%%%%%%
\begin{theorem}\label{th:new31}
The space $S_0^1(\lambda)$ is finite-dimensional for every $\lambda$.
The values of $\lambda$ for which $S_0^1(\lambda)$ is nontrivial satisfy
$\lambda= -(\frac{1}{4} + r^2)$ for some
$r>0$, corresponding to $s_0 = \frac{1}{2} \pm i r$ with $\lambda=s_0(s_0-1)$.
\end{theorem}
\begin{proof} 
This result  is a basic consequence of Selberg's theory. The condition
that $\lambda$ is real  follows from self-adjointness of the Maass Laplacian $\Delta_0$
on the modular surface. The fact that   $\lambda \leq 0$ follows from a positivity
property for $-\Delta_0$. 
The fact that there are no eigenvalues $-\frac{1}{4}<\lambda\leq0$ follows from a computation of Maass \cite{Maa83}.

The finite multiplicity follows from the a version the Weyl law for eigenvalues for $X(1)=\PSL(2,\Z)\backslash \H$.
Let $N(T)$ count the number of eigenvalues $s_0 = \frac{1}{2} + i r_j$ having $|r_j| \le T$ 
(with multiplicity).
Then we have
\[
	N_{X(1)}(T) = \frac{1}{12} T^2 +O ( T \log T),
\]
see \cite[Sec. 15.5]{IK04}.
\end{proof}

The space $S_0^1(\lambda)$  may be divided into subspaces of even and odd Maass forms:
even forms have Fourier coefficients $a_{-n}= a_{n}$ and odd forms have $a_{-n}= -a_n$. 
The coefficients $a_n$ are in general complex-valued. 
Maass shows that a theory of Hecke operators is applicable to the coefficients $a_n$  and
that  $S_0^1(\lambda)$ has a basis consisting of Hecke eigenforms. 

It is conjectured that the spectrum of cusp forms for $X(1)$  is simple \cite[Conjecture 3]{Sar03}, 
when given in
terms of the value of $s_0 = \frac{1}{2} + ir$ 
(the eigenfunctions occur in
complex conjugate pairs, with $\bar{s}_0= \frac{1}{2} -ir$, so that $\dim S_0^1(\lambda)$ is always even).
 Sarnak attributes  this conjecture to Cartier \cite{Cart71} (who raised it as Problem (B) in \cite[p. 39]{Cart71}).
It is supported by the existing computational evidence.
The following is the best upper bound  on the multiplicity of eigenvalues currently known. 
It is stated in Sarnak \cite[Sect. 4, (36)]{Sar03}, with a proof being
outlined in  \cite{Sar02}.

%%%%%%%%%%%%
% Theorem 2.3
%%%%%%%%%%%%
\begin{theorem}\label{th:new32}
Let $m_{X(1)}(\lambda)=\dim S_0^1(\lambda)$ denote the multiplicity of Maass cusp forms for $\PSL(2,\Z)$ with eigenvalue $\lambda$.
Then
\[
	\limsup_{|\lambda| \to \infty} m_{X(1)}(\lambda) \frac{ \log |\lambda|}{\sqrt{|\lambda|}} \le \frac{ \operatorname{Vol} (X(1))}{4} = \frac{\pi}{12}.
\]
\end{theorem}

There has been much work done on computing eigenvalues of Maass cusp forms.
The first few eigenvalues for the Dirichlet problem for Maass cusp forms for $\PSL(2,\ZZ)$ 
(``odd eigenvalues'') given in terms of $s =\frac{1}{2} + ir$ with associated eigenvalue
$\lambda= \frac{1}{4} + r^2$ were computed numerically
by Cartier \cite{Cart71} (actually by P.  Mignot, U. of Stasbourg) as
$r_1=9.47, r_2=12.13, r_4=14.34,...$, along with  unreliable computations for  ``even'' eigenvalues
including some spurious eigenvalues, later noted to be associated with some zeta and $L$-function
zeros, as explained in Hejhal \cite{Hej81}.
The first few ``odd'' and ``even'' eigenvalues  were computed reliably 
by Hejhal \cite[Table IV]{Hej81}, see also   \cite[p. 513]{Hej83}. They  are $r_1= 9.533695, r_2=12.17301, r_3=13.77975, r_4=14.35851$,
whence $\lambda = -\frac{1}{4}-r^2$ has values $\lambda_1 = -91.14..., \lambda_2 = -148.43 ..., \lambda_3 = -190.13...$
and $\lambda_4 = -206.16...$, see Sarnak \cite[p.444]{Sar03}.  
Here $r_3$ is an ``even'' eigenvalue. 
Further computations appear in Hejhal \cite{Hej91},
reprinted in \cite[pp. 125--165]{Hej92} and Steil \cite{Steil1994}.

It is known that  the spectral multiplicity of Maass waveforms on 
the congruence subgroup $\Gamma_0(N)$ for non-squarefree level $N$
can have multiplicity greater than one (Str\"{o}mberg \cite{Stro12}, Humphries \cite{Hump15}).
The prevailing belief is that the eigenvalue multiplicity should be bounded above on each $\Gamma_0(N)$.

A converse theorem for Maass forms on the full modular group was proved by Raghunathan \cite{Rag10}.
The Riemann hypothesis for $L$-functions of Maass wave forms for $\PSL(2, \ZZ)$ was tested
numerically for low zeros for a few small Maass parameters $s_0 = \frac{1}{2} \pm ir$ 
by Str\"{o}mbergsson \cite{Strn99}. A method of computing  Fourier coefficients of Hecke-equivariant Maass cusp forms 
was given by Stark \cite{Sta84}.

%%%%%%%%%%%%%%%%%%%%%%%%%%%%%%%%%%%%%%%%%%%%%
%%                                         %%
%%     2.2   Maass Eisenstein series       %%
%%                                         %%
%%%%%%%%%%%%%%%%%%%%%%%%%%%%%%%%%%%%%%%%%%%%%

\subsection{Maass's Eisenstein series}\label{sec:MaassSeries}

For later reference 
we record results and calculations of Maass
\cite[Chapter IV]{Maa83} concerned with the Eisenstein series 
\begin{equation}\label{eqn:maassEisensteinDef}
G(z, \ol{z}; \alpha, \beta) := \sum_{(m,n) \ne (0,0)} (mz+n)^{-\alpha} (m\bar{z} + n)^{-\beta}
= \sum_{(m,n) \ne (0,0)} (mz+n)^{-(\alpha-\beta)}|m{z} + n|^{-2\beta} ,
\end{equation}
in which $\im{z}>0$,  $\alpha - \beta\in 2\Z$, $\re{\alpha+\beta} > 2$,  and the 
sum runs over all pairs of integers not equal to $(0, 0)$. 
We define the function by the second summation, taking the principal
branch of the logarithm, noting that $|mz+n| = |m\bar{z}+n| >0$ under the hypotheses.%
\footnote{Maass wishes to view the variables $z$ and $\bar{z}$ as independent, and requires
a convention on 
the branch of the logarithm defining $(mz +n)^{-\alpha} = e^{-\alpha  \log(mz+n)}$, requiring
a compatibility condition to hold between $mz+n$ and $m\bar{z}+n$.}
Maass introduces the notation $q := \alpha + \beta$, $r := \alpha - \beta \in 2\ZZ$;
in terms of our  variables we have $q=2s+k, r=k$.\footnote{On page
209 Maass defines a parameter $k$
by $r=2k \in 2\ZZ$.  This notation conflicts with our notation; his $k$ equals our $k/2$. 
In  the statement of Theorem \ref{thm:maassFourierExpansion}
 we  replace the Maass parameter $k$ by $\frac{r}{2}$ wherever it occurs.}

%%%%%%%%%%%%
% Theorem 2.4
%%%%%%%%%%%%
\begin{theorem}[Maass \cite{Maa83} page 210]\label{thm:maassFourierExpansion}
{\rm (Fourier expansion for $G(z, \ol{z}; \alpha, \beta)$)}
For fixed  $\alpha - \beta \in 2\Z$, 
the  Fourier expansion  of the function $G(z, \ol{z}; \alpha, \beta)$
with respect to periodicity under $x \mapsto x+1$ is:
\[
G(z, \ol{z}; \alpha, \beta) = \varphi_{\frac{r}{2}}(y, q) + 2(-1)^{\frac{r}{2}} (\sqrt{2} \pi)^q 
\sum_{n \in \ZZ-\{ 0\}} \frac{\sigma_{q-1}(n)}{\Gamma \left(  \frac{q}{2} + \sgn{n} \frac r2 \right) } W(2\pi n y; \alpha, \beta) e^{2\pi i n x}.
\]
In this  formula the constant term is
\[
\varphi_{\frac{r}{2}}(y, q) := 2 \zeta(q) + (-1)^{\frac{r}{2}} (4 \pi)2^{1-q}  \frac{\Gamma( q-1)}{\Gamma \left(  \frac{q+r}{2}  \right)  \Gamma \left(  \frac{q-r}{2}  \right)  } \zeta(q-1) y^{1-q}.
\]
This formula also  contains the $s$-th power divisor function $\sigma_{s}(n) := \sum\limits_{d\mid n, d>0} d^{s}$,
for  $s \in \CC$, and
the modified Whittaker function, defined  for $y>0$ by %with $\epsilon(n) =\sgn{n}= \pm 1$,
\[
	W(\pm y; \alpha, \beta) := y^{-\frac{q}{2}} W_{\pm \frac{r}{2}, \frac{1}{2} (q-1)} (2y),
\]
where $W_{\kappa, \mu}(y)$ is the $W$-Whittaker function (see Section~\ref{sec:newsec4}).
\end{theorem}

\begin{proof}
The main formula appears in \cite[p.210]{Maa83},
with constant term written 
$$\varphi_{\frac{r}{2}}(y, q) := 2 \zeta(q) + (-1)^{\frac{r}{2}} 2^{3-q} \pi \frac{\Gamma( q-1) \zeta(q-1)}{\Gamma \left(  \frac{q+r}{2}  \right)  \Gamma \left(  \frac{q-r}{2}  \right)  }\big((1-q)u(y, q) +1\big).$$
% where for clarity we have set Maass's $k= \frac{r}{2}$. 
We use $u(y,q) = \frac{y^{1-q}-1}{1- q}$ if $q \ne 1$
(defined on \cite[p. 181]{Maa83} ) whence  
$(1-q)u(y,q) +1 = y^{1-q}.$
\end{proof}

This Eisenstein series  was initially defined for $\re{\alpha + \beta}>2$ 
but the Fourier expansion  yields an analytic continuation of $G(z, \ol{z}; \alpha, \beta)$
in $q=\alpha+ \beta$  to all $q \in \CC$,
using the fact that $W_{\kappa, \mu}(y)$ is an entire function of $\mu$ for fixed $\kappa, y$,
and has rapid decay as $y \to \infty$.

Maass \cite[p. 213]{Maa83} introduced 
completed
versions of 
these Eisenstein series, setting
 \begin{equation}\label{Mdouble}
G^*(z, \ol{z}; \alpha, \beta) := \mfrac{q}{2} \left(  1- \mfrac{q}{2} \right)  \pi^{-\frac{q}{2}} \Gamma \left(  \mfrac{q}{2} + \mfrac{|r|}{2} \right)  G(z, \ol{z};\alpha, \beta)
\end{equation}
and defining its 
completed
Fourier series  constant term by
\begin{equation}
\varphi_{\frac{r}{2}}^*(y, q) :=  \mfrac{q}{2} \left(  1- \mfrac{q}{2} \right)  \pi^{-\frac{q}{2}} \Gamma \left(  \mfrac{q}{2} + \mfrac{|r|}{2} \right)   \varphi_{\frac{r}{2}}(y,q).
\end{equation}
These functions have the following properties.

%%%%%%%%%%%%
% Theorem 2.5
%%%%%%%%%%%%
\begin{theorem}\label{thm:maassProperties}
{\rm (Properties of $G(z, \ol{z}; \alpha, \beta)$)}
For fixed $r= \alpha-\beta \in 2\ZZ$, the functions $G(z, \ol{z}; \alpha, \beta)$ 
and their completions have the following
properties.  
\begin{enumerate}[\textup(1\textup)]
\item
{\em (Analyticity in $q$-variable)}
The completed function 
$G^*(z, \ol{z}; \alpha, \beta)$ and its completed constant term $\varphi_{r/2}^*(y, q)$ 
both analytically continue  to  entire functions of  the variable  $q = \alpha + \beta$. 
\item 
{\rm (Laplacian Eigenfunction)} 
In the variables
$z=x+iy$, $\bar{z}=x-iy$ it is
a solution of the (elliptic) partial differential equation
$$\Omega_{\alpha,\beta} G(z, \ol{z}; \alpha, \beta) = 0$$
where 
\begin{equation}
\Omega_{\alpha, \beta}= -y^2  \left(  \frac{\partial^2}{\partial x^2}+  \frac{\partial^2}{\partial y^2} \right)  
 +i(\alpha- \beta) y \frac{\partial}{\partial x} - (\alpha +\beta ) y \frac{\partial}{\partial y}.
\end{equation}
The function $G(z, \ol{z}; \alpha, \beta)$ is real-analytic in the variables $x, y$ on $\HH$.
These properties also hold for $G^{\ast}(\ol{z}; \alpha, \beta).$
\item
 {\rm ($\{\alpha, \beta\}$-Modular Invariance)}
For each $\gamma = \smatr{a}{b}{c}{d} \in \SL(2,\Z)$ we have
\begin{equation}
G\left( \frac{az+b}{cz+d}, \frac{a\bar{z}+b}{c \bar{z}+ d}; \alpha, \beta\right)=  (cz+d)^{\alpha- \beta}|cz+d|^{2\beta}
G(z, \ol{z}; \alpha, \beta). 
\end{equation}
\item {\em (Moderate Growth)} For some constant $K>0$, $G(z, \ol{z}; \alpha, \beta) = O(y^K)$ for $y\to \infty$, uniformly in $x$. 
\item
{\em (Functional Equations)} $G^*(z, \ol{z}; \alpha, \beta)$ satisfies the following: 
$$G^*(z, \ol{z}; \ol{\beta}, \alpha) = G^*(-\ol{z}, -z; \alpha, \beta)$$
and 
$$G^*(z, \ol{z}; 1-\alpha, 1-\beta) = y^{q-1} G^*(z, \ol{z}; \beta, \alpha).$$
\end{enumerate}
\end{theorem}
\begin{proof}
Results (1), (3), (4), and (5) are given on \cite[pp. 213--214]{Maa83} (in (5) we have corrected a typo in the first equation). 
In particular (3) and (4) are  obtained by unwinding the definition of 
$[\SL(2,\Z), \alpha, \beta, 1]$, the space of automorphic functions of type $\{ \SL(2,\Z), \alpha, \beta, 1\}$ on \cite[p. 185]{Maa83}. 
The differential operator  in (2) is defined on \cite[pp. 175--176]{Maa83},
and that   (2) holds  for $G(z, \ol{z}; \alpha, \beta)$
appears on \cite[pp. 177, equation (14) and preceding equation]{Maa83}.
The result  (2)  then holds for $G^{\ast}(z, \ol{z}; \alpha, \beta)$ because the prefactor
is a constant with respect to the differential equation and for real-analyticity.
\end{proof}

% NA: It seemed like a good idea to just write these in the Theorem since that is how they are used later.
% We note that the operator $\Omega_{\alpha\beta}$ can be rewritten
% \begin{equation}
% \Omega_{\alpha, \beta}= -y^2  \left(  \frac{\partial^2}{\partial x^2}+  \frac{\partial^2}{\partial y^2} \right)  
%  +i(\alpha- \beta) y \frac{\partial}{\partial x} - (\alpha +\beta ) y \frac{\partial}{\partial y}.
% \end{equation}
% In addition the $\{ \alpha, \beta\}$-modular invariance property  under $\gamma = \smatr{a}{b}{c}{d} \in \SL(2,\Z)$  in
% Theorem \ref{thm:maassProperties} (3) 
% can be rewritten
% % $G(z, \ol{z}; \alpha, \beta)$ transforms under $SL(2, \ZZ)$ as (\cite[p. 169]{Maa83})
% \begin{equation}
% G( \frac{az+b}{cz+d}, \frac{a\bar{z}+b}{c \bar{z}+ d}; \alpha, \beta)=  (cz+d)^{\alpha- \beta}|cz+d|^{2\beta}
% G(z, \ol{z}; \alpha, \beta). 
% \end{equation}

%%%%%%%%%%%%%%%%%
%%                                         %%
%%     2.3   Non-holomorphic Eisenstein series       %%
%%                                         %%
%%%%%%%%%%%%%%%%%

\subsection{Properties of  non-holomorphic Eisenstein series}\label{sec:NHE}

 Recall that the weight $k$ non-holomorphic Eistenstein series $E_k(z, s)$ is defined as
\begin{equation}\label{Eis}
E_{k}(z, s) :=\mfrac{1}{2} \sum_{(m,n) \in \ZZ^2 \backslash  (0,0)} \frac{y^s}{|mz+n|^{2s} (mz+n)^{k}}.
\end{equation}
This series converges absolutely for  $\re s  > 1-\frac k2$. The resulting function 
transforms under  elements of $\SL(2,\ZZ)$ as 
\[
	E_{k}\left(\frac{az+b}{cz+d}, s\right) =  (cz+d)^{k} E_k(z, s).
\]
At $s=0$ for  even weights $k \ge 4$
this function specializes to  an
unnormalized version of the  holomorphic Eisenstein series $E_k(z)$,  with
\[
	E_{k}(z, 0) = \mfrac{1}{2} \sum_{(m,n) \in \ZZ^2 \backslash \{0\}} \frac{1}{(mz+n)^k}
	= \frac{1}{2} \zeta(k) E_k(z).
\]
With our scaling $E_{k}(z, 0) = \frac{1}{2}G_k(z)$,
where $G_k(z)$ is the usual unnormalized holomorphic Eisenstein series  in 
Serre \cite{Se73}.

%%%%%%%%%%%%%%%%%%%%%%%%%%%%%%%%%%%%%%%%%%%%%%%%%%%%%%%%%%%%%%%%%%%%%%%%%%%%%%
% It thus gives a continuous non-holomorphic
%family containing this holomorphic modular form at one value.
%%%%%%%%%%%%%%%%%%%%%%%%%%%%%%%%%%%%%%%%%%%%%%%%%%%%%%%%%%%%%%%%%%%%%%%%%%%%%%%

The non-holomorphic Eisenstein series  $E_k(z, s)$ has a simple relation to Maass's
Eisenstein series
 $G(z, \bar{z}; \alpha, \beta)$. 
Take $\alpha= s+k$ and $ \beta=s$ with $k \in 2\ZZ$. 
Then in  the absolute convergence region $\re{s} >1-\frac k2$ we  have the identity 
\begin{equation}\label{E-to-G}
E_{k}(z, s) =\mfrac{1}{2} y^{s} G(z, \ol{z}; s+k, s),
\end{equation}
which then  holds   for  all $s \in \CC$    under  analytic continuation.

%%%%%%%%%%%%%%%%%%%%%%%%%%%%%%%%%%%%%%%%%%%%%%%%%%%%%%%%%%%%%%%%%%%%%%%%%%%%%%%%%%
%This relation implies that for fixed $s$,  
%the functions $E_k(z, s)$ and $E_k(z, 1-k-s)$ as functions of $z$ are linearly dependent,
%leading to the (well-known)  functional equation of shape 
%$E_k(z, s) = \gamma(s) E_k(z, 1-k-s)$,
%with central line $Re(s) = \frac{1-k}{2}$.
%%%%%%%%%%%%%%%%%%%%%%%%%%%%%%%%%%%%%%%%%%%%%%%%%%%%%%%%%%%%%%%%%%%%%%%%%%%%%%%%%%

One can  use  Maass's  result for $G(z, \ol{z}; \alpha, \beta)$ to obtain Fourier expansions for 
the {\em completed non-holomorphic Eisenstein series} $\htE_k(z, s)$ defined by 
\begin{equation}\label{completed}
\htE_k(z, s) := \pi^{-s-\frac{k}{2}}\Gamma\left(s+ \mfrac{k}{2} + \mfrac{|k|}{2}\right) E_k(z, s),
\end{equation}
as follows.
We denote by $\htz(s)$ the completed Riemann zeta function
\[
	\htz(s) := \pi^{-s/2} \Gamma(s/2) \zeta(s).
\]

\begin{proposition}\label{prop:FourierExpansionArbitrary}
{\rm (Fourier expansion of $\htE(z,s)$)\footnote{{The corresponding Proposition 3.5 of \cite[p. 304]{LR15K} has misprints. In the nonconstant
terms with $e^{2\pi i n x}$  the multiplicative factor $(\sqrt{2 }\pi)^{2s+k} \pi^{-s-\frac{k}{2}}$ should read $(2\pi)^{\frac{k}{2}} |n|^{-s}$
as in Proposition \ref{prop:FourierExpansionArbitrary} here.}}}
For $k\in 2\ZZ$,
the completed non-holomorphic Eisenstein series 
$\htE(z, s)$
has the Fourier expansion 
\begin{multline*}
\htE_k(z,s) = C_{0}(y,s) 
 + (-1)^{\frac{k}{2}} \Gamma \left( s+\mfrac{k}{2}+\mfrac{|k|}{2} \right)  y^{-\frac k2}   \\
\times \sum_{n\in \Z \backslash \{0\}} 
\frac{|n|^{-s-\frac k2}\sigma_{2s+k-1}(n)}{\Gamma \left(  s+\frac{k}{2}(1 + \sgn{n})  \right)  } 
 W_{\sgn{n} \frac k2 , s+\frac{k-1}{2} } (4\pi|n|y)e^{2\pi i n x},
\end{multline*}
in which the Fourier constant term is
\begin{multline*}
C_0(y,s) = \\ \frac{\Gamma \left( s+\frac{k}{2} + \frac{\abs{k}}{2} \right) }{\Gamma \left( s+\frac{k}{2} \right) }  \htz(2s+k)\, y^s
+  (-1)^{\frac{k}{2}} \frac{\Gamma \left(  s+ \frac{k}{2} \right)  \Gamma  \left(  s+ \frac{k}{2} + \frac{\abs{k}}{2} \right) }{\Gamma( s+ k) \Gamma(s) } \htz(2-2s-k)\, y^{1-s-k}
\end{multline*}
 and $\sigma_s(n) = \sum\limits_{d |n, d>0} d^{s}.$
\end{proposition}

\begin{proof}
The Fourier expansion follows from that of $G(z, \ol{z}; \alpha, \beta)$ in Theorem \ref{thm:maassFourierExpansion}
via the relation
\[
\htE_k(z,s) =  \pi^{-s-\frac{k}{2}} \Gamma  \left(  s+ \mfrac{k}{2} +\mfrac{\abs{k}}{2} \right) 
\cdot \mfrac{1}{2} y^s G(z, \overline{z}; s+k, s).
\]
The duplication formula for the gamma function is used, 
as well as the functional equation of the Riemann zeta function   $\htz(2s+k-1) =\htz(2-2s-k)$.
\end{proof}

For later use we give a more detailed formula for the constant term in the Fourier series.
\begin{proposition}\label{prop:37}
For $k \in 2\ZZ$ and $s \in \CC$ the  Fourier constant term of the completed Eisenstein series $\htE_k(z, s)$ is as follows.
\begin{enumerate}[\textup(1\textup)]
\item Suppose $s \ne \frac{1-k}{2}$. Then for weights $k \ge 2$ we have
	\begin{multline*}
	C_0(y,s) =  \left( s+\mfrac{k}{2} \right)   \left( s+ \mfrac{k+2}{2} \right)  \cdots  \left( s+ \mfrac{2k-2}{2} \right)  \htz(2s+k)\, y^s \\
	+  (-1)^{\frac{k}{2}} s(s+1) \cdots  \left( s+ \mfrac{k-2}{2} \right)  \htz(2-2s-k)\, y^{1-s-k}.
    \end{multline*}
For weights $k \le  -2$, we have
	\begin{multline*}
	C_0(y,s) = (s-1)(s-2) \cdots  \left( s- \mfrac{|k|}{2} \right)  \htz(2s+k)\, y^s \\
	+  (-1)^{\frac{k}{2}} (s-|k|) (s- |k| +1) \cdots  \left( s- \mfrac{|k|}{2} -1 \right) \htz(2-2s-k)\, y^{1-s-k}.
    \end{multline*}
For weight $k=0$ and $s \ne 0$ or $1$, $C_0(y, s) = \htz(2s)y^s + \htz(2-2s) y^{1-s}$.

\item Suppose $s = \frac{1-k}{2}$. Then for weights $k \ne 0$,
$$
C_0 \left( y, \mfrac{1-k}{2} \right)  = \mfrac{1}{2} \cdot \mfrac{3}{2} \cdots \mfrac{ |k|-1}{2}
\left( \gamma - \log 4\pi + \log y + 2\left( 1 + \tfrac 13 + \cdots + \tfrac{1}{|k|-1} \right) \right)y^{\frac{1-k}{2}},
$$
where $\gamma$ is Euler's constant.
For weight $k=0$,
$C_0(y, \tfrac{1}{2}) =  y^{\frac{1}{2}} \left( \gamma - \log 4\pi + \log y \right)$.
\end{enumerate}
\end{proposition}

Note that Proposition \ref{prop:37} immediately gives the relation $C_0(y, s) = C_0(y, 1-k-s)$.

\begin{proof}[Proof of Proposition \ref{prop:37}]
(1)  The products of Gamma factors in
the constant term formula in \eqref{prop:FourierExpansionArbitrary}
simplify to polynomials using
the identity  $s \Gamma(s) =\Gamma(s+1)$.

(2)
We write $s= \frac{1-k}{2} + \frac{\epsilon}2$ and find that
\begin{equation*}
	C_0\left(y,\mfrac{1-k}2+\mfrac \epsilon2\right) = \Gamma\left(\mfrac{1+|k|}{2}+\mfrac\epsilon2\right) y^{\frac{1-k}2} \left( \frac{\hat\zeta(1+\epsilon)}{\Gamma(\frac 12+\frac \epsilon2)} y^{\frac \epsilon 2} + (-1)^{\frac k2} \frac{\Gamma(\frac 12+\frac \epsilon 2)\hat\zeta(1-\epsilon)}{\Gamma(\frac{1+k}{2}+\frac\epsilon 2)\Gamma(\frac{1-k}{2}+\frac\epsilon 2)} y^{-\frac \epsilon 2} \right).
\end{equation*}
Let $\psi(s)=\frac{\Gamma'}{\Gamma}(s)$.
Using the series exapnsions
\begin{align*}
	\Gamma(a+\epsilon)&=\Gamma(a)(1+\psi(a)\epsilon+O(\epsilon^2)), \\
	\hat\zeta(1+\epsilon) &= \frac{1}{\epsilon} + \mfrac 12(\gamma-\log 4\pi) + O(\epsilon), \\
	y^{\epsilon} &= 1 + (\log y) \epsilon + O(\epsilon^2),
\end{align*}
and the reflection formulas $\Gamma(s)\Gamma(1-s)=\frac{\pi}{\sin\pi s}$ and $\psi(s)-\psi(1-s)=\frac{-\pi}{\tan\pi s}$, we find that
\begin{equation*}
	C_0\left(y,\mfrac{1-k}2+\mfrac \epsilon2\right) = \mfrac{1}{\sqrt \pi} \, \Gamma\left(\mfrac{1+|k|}{2}+\mfrac\epsilon2\right) \left(2\gamma-\log \pi+\log y + \psi\left(\mfrac{1+k}{2}\right)+O(\epsilon)\right)y^{\frac{1-k}2}.
\end{equation*}
When $k=0$ we have $\Gamma(\frac{1+|k|}{2})=\sqrt \pi$ and $\psi(\frac{1+k}{2}) = -\gamma-\log 4$.
For $k\neq 0$ we use \cite[(5.4.2) and (5.5.4)]{NIST} to obtain
\begin{align*}
	\mfrac{1}{\sqrt \pi}\Gamma\left(\mfrac{1+|k|}{2}\right) &= \mfrac{1}{2} \cdot \mfrac{3}{2} \cdots \mfrac{ |k|-1}{2}, \\
	\psi\left(\mfrac{1+k}{2}\right) &= - \gamma - \log 4 + 2 \left( 1 + \tfrac 13 + \cdots + \tfrac{1}{|k|-1} \right).
\end{align*}
The proposition follows.
\end{proof}

We now collect various analytic properties of the Eisenstein series.

\begin{theorem}\label{th38}
{\rm (Properties of $E_k(z,s)$)}
Let $k \in 2\ZZ$.

\begin{enumerate}[\textup(1\textup)]
\item {\rm (Analytic Continuation)} For fixed $z \in \HH$, the  completed weight $k$ Eisenstein series  $\htE_k(z,s)$
analytically continues to the $s$-plane as
 a meromorphic function. For $k=0$ 
its has two singularities, which are 
simple poles at $s=0$ and $s=1$
with residues $-\frac{1}{2}$ and $\frac{1}{2}$, respectively.
For $ k \ne 0$ it is an entire function.

\item {\rm (Functional Equation)} For fixed $z \in \HH$, the completed weight
$k$ Eisenstein series
satisfies the functional equation
\begin{equation}~\label{213aa}
\htE_k(z, s)= \htE_k(z, 1-k-s).
\end{equation}
The doubly-completed series
\begin{equation} \label{Ek-dc-def}
	\dhtE_k(z,s) :=\left(s+ \mfrac{k}{2}\right)\left(s +\mfrac{k}{2}-1\right) \htE_k(z,s)
\end{equation}
 is an entire function of $s$ for all $k \in 2\ZZ$
and  satisfies the same functional equation
\begin{equation}~\label{213b}
\dhtE_k(z, s)= \dhtE_k(z, 1-k-s).
\end{equation}
The center line of these functional equations is $\re{s} = \frac{1-k}{2}.$

\item {\rm ($\Delta_k$-Eigenfunction)} $E_k(z, s)$ is a (generalized) eigenfunction
of the hyperbolic
Laplacian operator $\Delta_k$ with eigenvalue $\lambda= s(s+k-1).$ 
That is, for all $s \in \CC$,
\begin{equation}\label{209a}
\Delta_k E_{k}(z, s) = s(s+k-1) E_{k}(z, s).
\end{equation}
This eigenfunction property holds for 
the completed functions $\htE_k(z, s)$ and $\dhtE_k(z, s).$
\end{enumerate}
\end{theorem}

\begin{proof} 
(1) and (2).
We have $E_k(z, s) = y^s G(z, \ol{z}; s+k, s)$.
For the doubly-completed Eisenstein series, a calculation 
from the definition \eqref{Mdouble}
verifies the identity
\begin{equation}\label{E-to-G2}
\dhtE_{k}(z, s) = -\mfrac{1}{2} y^{s} G^{\ast}(z, \ol{z}; s+k, s).
\end{equation}
The latter is an entire function by Theorem \ref{thm:maassProperties}(1).
The functional equation \eqref{213b} for $\dhtE_k(z, s)$
follows from unwinding the second functional equation  
$$G^*(z, \ol{z}; 1-\alpha, 1-\beta) = y^{q-1} G^*(z, \ol{z}; \beta, \alpha)$$
in Theorem \ref{thm:maassProperties} (4), 
now making the assignment $\alpha=s, \beta=s+k$, so $q=2s+k$, and multiplying both sides by $y^{1-k-s}$.
Now the relation \eqref{Ek-dc-def}
implies  that $\htE_k(z, s)$ is meromorphic in $s$ and that  its only possible  singularities 
are at most simple poles at  $s= -\frac{k}{2}$ and at $s= 1-\frac{k}{2}$.
It inherits the functional equation $\htE_k(z, s) = \htE_k(z, 1-k-s)$, 
which  interchanges the two possible polar points and  implies the sum of
the residues of the (possible) poles at these two points is $0$.

It remains to determine when poles are present for general $k \in 2\ZZ$
which can be done using  its Fourier series expansion given
in Proposition \ref{prop:FourierExpansionArbitrary}.
The only singularities in $s$ can be contributed by
 the constant term of the Fourier expansion.
The case $k=0$ is well treated in the literature; 
where it has poles at $s=0$ and $s=1$
with residues as specified, e.g. \cite[pp. 45--46]{Kub73}, \cite[pp. 98--100]{LS06}.
For all $k \ne 0$  the gamma factors in the
formula for the constant term $C_0(y,s)$ contribute a canceling zero at all
the possible pole location points $s= \frac{2-k}{2}, s= \frac{k}{2}$ and $s= \frac{1-k}{2}$
associated to $\htz(2s+k)$ and $\htz(2-2s-k)$, and the functions 
$\htE_k(z, s)$ are entire functions. 

(3) From Theorem \ref{thm:maassProperties}(2) we have
\[
\Omega_{s+k, s} G(z, \ol{z}; s+k, s)=0,
\]
where
\begin{equation}\label{MPDE}
\Omega_{s+k, s}  = -y^2  \left(  \frac{\partial^2}{\partial x^2}+  \frac{\partial^2}{\partial y^2} \right)  
 +i k y \frac{\partial}{\partial x} - (k+2s) y \frac{\partial}{\partial y}=
 - \Delta_k - 2s y \frac{\partial}{\partial y}.
\end{equation}
By \eqref{E-to-G} it follows that
\[
	\big(\Delta_k - s(s+k-1)\big) E_k(z, s) = 0. \qedhere
\]
\end{proof}

%%%%%%%%%%%%%%%%%%%%%%%%%%%%%%%%%%%%%%%%%%%%%
%%                                         %%
%%     3  POLYHARMONIC                     %%
%%                                         %%
%%%%%%%%%%%%%%%%%%%%%%%%%%%%%%%%%%%%%%%%%%%%%

\section{Polyharmonic Fourier Series Expansions} \label{sec:newsec4}

In this section we give the general form of the Fourier expansions of shifted polyharmonic Maass forms, with and without growth restrictions.
These expansions involve the Whittaker functions and their derivatives in the auxiliary parameter $s$.
We begin by discussing a certain family of linearly independent solutions to Whittaker's differential equation.

\subsection{Linearly independent Whittaker function solutions}\label{subsec:A1}
Given parameters $\kappa, \mu \in \C$, 
the Whittaker differential equation \cite[(13.14.1)]{NIST} is
\begin{equation} \label{eq:whit-diffeq}
\frac{d^2F}{dz^2}  + \left(-\frac{1}{4} + \frac{\kappa}{z} + \frac{\frac{1}{4} - \mu^2}{z^2}\right)F =0.
\end{equation}
The standard solutions to \eqref{eq:whit-diffeq} are the $M$-Whittaker function $M_{\kappa,\mu}(z)$ and the $W$-Whittaker function $W_{\kappa,\mu}(z)$.
For all parameters $(\kappa, \mu) \in \CC^2$ 
the  Whittaker $W$-function $W_{\kappa, \mu}(z)$ is uniquely determined up to a multiplicative constant
by its property of having rapid decay along the
positive real axis. 
The function $M_{\kappa,\mu}(z)$ does not exist for $2\mu+1 \in \Z_{\leq 0}$; furthermore, its Wronskian with $W_{\kappa,\mu}(z)$ is given by \cite[(13.14.26)]{NIST}
\[
	\mathcal W \left\{ M_{\kappa,\mu}(z), W_{\kappa,\mu}(z) \right\} = -\frac{\Gamma(1+2\mu)}{\Gamma(\frac 12 +\mu-\kappa)},
\]
so it is linearly dependent on $W_{\kappa,\mu}$ whenever $\frac 12+\mu-\kappa \in \Z_{\leq 0}$.

A second independent solution of the Whittaker differential equation is obtained by analytic continuation of
the function $W_{\kappa,\mu}(z)$ in the $z$-variable to the negative real axis. The analytic continuation of the function is multi-valued
with a branch point at $z=0$, so we get two different functions, according as we continue along a path in
the upper half plane,
\begin{equation} \label{eq:W-plus-def}
\sM^{+}_{\kappa, \mu}(z) := W_{-\kappa, \mu} (z e^{\pi i}),
\end{equation}
or along a path in the lower half plane,
$$
\sM^{-}_{\kappa, \mu}(z) := W_{-\kappa, \mu} (z e^{- \pi i}).
$$
The notation $\sM_{\kappa,\mu}^\pm(z)$ is introduced here and is not standard.
The functions $\sM_{\kappa,\mu}^\pm(z)$ are linearly independent with $W_{\kappa,\mu}(z)$ since the Wronskian \cite[(13.14.30)]{NIST}
\[
{\mathcal W}\{ W_{\kappa, \mu}(z), W_{-\kappa, \mu} ( z e^{\pm \pi i}) \} = e^{\mp \pi i \kappa}
\]
is everywhere nonzero.
For the sake of concreteness, in this paper we always choose $\sM_{\kappa,\mu}^+(z)$ as the second linearly independent solution to \eqref{eq:whit-diffeq}.

\subsection{Shifted harmonic Fourier coefficients}\label{sec:FourierExp1}

The following well-known result gives
allowable functional forms of Fourier coefficients for periodic functions
that satisfy $(\Delta_k- \lambda) f_n(z) = 0$,
with no restriction on the growth of $f_n$ at the cusp.

%%%%%%%%%%%%%%%%%%%
% THEOREM 3.1
%%%%%%%%%%%%%%%%%%%
\begin{theorem}\label{thm:41a}
{\rm (Shifted harmonic Fourier coefficients of unrestricted growth)}
Let  $k \in 2 \ZZ$ and suppose  that $f_n(z) = h_n(y) e^{2 \pi i nx}$  
satisfies
$$
(\Delta_k- \lambda) f_n(z) =0 \quad \mbox{for all} \quad z=x+iy \in \HH.
$$
Write $\lambda = s_0(s_0+k-1)$ for some $s_0 \in \CC$ (there are generally two choices for $s_0$).
Then the  complete set of such functions $h_n(y)$ are given in the following cases.
\begin{enumerate}
\item[\textup(1\textup)]
If $n \neq 0$ then
\[
h_n(y) = y^{-\frac{k}{2}}  
 \left(a_n^{-} \, W_{\sgn{n}\frac{k}{2}, s_0+ \frac{k-1}{2}}(4 \pi |n| y) 
 + a_n^{+} \, \sM^+_{\sgn{n}\frac{k}{2}, s_0+ \frac{k-1}{2}} (4 \pi |n| y) \right)
\]
for some constants $a_n^{-}, a_n^{+} \in \CC$, 
where $W_{\kappa,\mu}$ and $\sM_{\kappa,\mu}^+$ are the Whittaker functions discussed in Section~\ref{subsec:A1}.

\item[\textup(2a\textup)]
Suppose $n =0$ with $s_0 \ne \frac{1-k}{2}$ (equivalently, $\lambda \ne -(\frac{1-k}{2})^2$). Then
$$ 
h_n(y) = a_0^{-} \, y^{1-k-s_0} + a_0^{+} \,  y^{s_0},
$$
for some constants $a_0^{-}, a_0^{+} \in \CC$.
\item[\textup(2b\textup)]
Suppose $n =0$ with $s_0 = \frac{1-k}{2}$ (equivalently, $\lambda = -(\frac{1-k}{2})^2$).
Then
$$
h_n(y) =  a_0^- \, y^{\frac{1-k}{2}} + a_0^+ \, \mfrac{\partial}{\partial s} y^{s} \big|_{s=\frac{1-k}{2}} = a_0^- \, y^{\frac{1-k}{2}} + a_0^+ \, y^{\frac{1-k}2} \log y,
$$
for some constants $a_0^-, a_0^+\in \CC$.
\end{enumerate}
\end{theorem}

\begin{proof}
($1$).
Let $\epp \in \{-1,1\}$.
The requirement that  $(\Delta_k -\lambda)h_{\epp}(y) e^{\frac{\epp ix}{2}}=0$
leads to a second-order linear ordinary differential equation that $h_{\epp}(y)$ must satisfy which
simplifies if we set $h_{\epp}(y) = g_{\epp} (y) y^{-\frac{k}{2}}$. 
We obtain
\begin{equation*}
\Delta_k  \left(  g_{\epp}(y) y^{-\frac{k}{2}} e^{\frac{\epp ix}{2}}  \right)  = 
y^{-\frac k2} e^{\frac{i\epp x}{2}} \left( -\mfrac 14 \left( \epp^2 y^2 - 2\epp k y + k^2-2k \right) g_\epp(y) + y^2 g_\epp''(y) \right),
\end{equation*}
which leads to the differential equation
\[
g_{\epp}^{''}(y)  + 
\left(- \frac{1}{4} + \frac{\frac{\epp k}2}{y}+
\frac{ -\frac{k^2}4 + \frac k2 - \lambda}{y^2} \right) g_{\epp}(y) = 0.
\]
We now set $\lambda= s_0(s_0+k-1)$, which yields
$-\frac{k^2}4+\frac k2-\lambda = \frac14-(s_0+\frac{k-1}2)^2$,
from which we conclude that $g_{\epp}(y)$ satisfies Whittaker's differential equation
$$
F''(w) +  \left(  -\frac{1}{4}+ \frac{\kappa}{w} + \frac{\frac{1}{4} -\mu^2}{w^2}  \right) F(w) =0
$$
with parameters $(\kappa, \mu) = ( \frac{\epp k}{2}, s_0 + \frac{k-1}{2}).$
By the discussion in Section \ref{subsec:A1}, for a fixed $s_0$ we obtain two linearly independent solutions 
$$
f_{\epp}^{+}(z) = y^{- \frac{k}{2}} W_{\frac{\epp k}{2}, s_0 + \frac{k-1}{2}}(y) e^{\frac{\epp i}{2} x},
$$
and
$$
f_{\epp}^{-}(z) = y^{- \frac{k}{2}} \sM^+_{\frac{\epp k}{2}, s_0 + \frac{k-1}{2}}(y) e^{\frac{\epp i}{2} x}.
$$
We obtain the given solutions by  replacing  $z$ with $ 4 \pi |n| z$, noting that the operator $\Delta_k$ is invariant
under rescaling by a multiplicative constant. 

($2a$) and ($2b$).
For $n=0$, again writing $h_0(y)=g_0(y)y^{-\frac k2}$, we obtain instead the differential equation
$$
g_{0}^{''}(y)  + 
\frac{ \frac{1}{4} - (s_0+ \frac{k-1}{2})^2}{y^2} g_{0}(y) = 0,
$$
which has 
two linearly independent solutions
$y^{s+\frac k2}$ and $y^{1-\frac k2-s}$ provided $s \ne \frac{1-k}{2}$;
this gives (2a).
In the exceptional case $s= \frac{1-k}{2}$ the 
differential equation
has  two linearly independent solutions given by $y^{\frac{1}{2}}$ and $y^{\frac{1}{2}} \log y,$
as given in (2b).
\end{proof}

We next treat the special case of Fourier coefficients for shifted harmonic functions having moderate
growth at the cusp.
In Appendix~\ref{sec:AppA} (see \eqref{eq:W-asymp}) we show that $W_{\kappa,\mu}(y)$ decays exponentially as $y\to \infty$, 
while $\sM^+_{\kappa,\mu}(y)$ grows exponentially as $y\to\infty$.
This immediately gives the following theorem.

%%%%%%%%%%%%%%%%%%%
% THEOREM 3.2
%%%%%%%%%%%%%%%%%%%
\begin{theorem}\label{thm:41b}
{\rm (Shifted harmonic Fourier coefficients of moderate growth)}
Let  $k \in 2 \ZZ$ and suppose  that $f_n(z) = h_n(y) e^{2 \pi i nx}$  
satisfies
$$
(\Delta_k- \lambda) f_n(z) =0 \quad \mbox{for all} \quad z=x+iy \in \HH,
$$
and has at most polynomial growth in $y$ approaching the cusp $i\infty$.
Then the complete set of such $h_n(y)$ 
are those functions in Theorem \ref{thm:41a} which omit $\sM^+_{\kappa, \mu}(z)$,
i.e. those for which $a_n^{+} =0$ for all $n \ne 0$.
\end{theorem}

\subsection{Shifted polyharmonic Fourier coefficients of unrestricted growth}\label{sec:FourierExp2}

We now characterize for general $m \ge 2$ the 
individual Fourier coefficients $f_n(z)$ which satisfy $(\Delta_k- \lambda)^m f_n(z) = 0$,
with no  growth restriction at the cusp. 
The space of allowable  Fourier coefficient functions always has dimension $2m$.
The basic mechanism leading to new functions is
that the operator 
$\frac{\partial}{\partial s}$ commutes with $\Delta_k$
but does not commute with the multiplication operator  $\lambda I=s(s+k-1)I$ when $s$ is regarded
as variable.  
Letting $[A, B] = AB-BA$,  the commutation relation
$[sI, \frac{\partial}{\partial s}] = -I$  
yields
\begin{equation} \label{eq:commutation}
\left[\lambda I, \mfrac{\partial}{\partial s} \right] = \left[ s(s+k-1)I, \mfrac{\partial}{\partial s}\right] = ( 1-k - 2s)I.
\end{equation}

We first treat the case of the non-constant Fourier coefficients,
which involve 
derivatives in the second parameter
of the Whittaker functions discussed in Section~\ref{eq:whit-diffeq}.
We introduce a new notation for these functions.
For $n\neq 0$ and for each $m\geq 0$ define
\begin{align}
	\label{eq:u-neg-def}
	u_{k,n}^{[m],-}(y;s_0) &:=
	\begin{dcases} 
		y^{-\frac k2} \frac{\partial^m}{\partial s^m} W_{\sgn{n}\frac k2,s+\frac{k-1}2}(4\pi|n|y) \Big|_{s=s_0} & \text{ if } s_0\neq \mfrac{1-k}{2}, \\
		y^{-\frac k2} \frac{\partial^{2m}}{\partial s^{2m}} W_{\sgn{n}\frac k2,s+\frac{k-1}2}(4\pi|n|y) \Big|_{s=\frac{1-k}{2}} &\text{ if }s_0=\mfrac{1-k}{2},
	\end{dcases}
	\\
	\label{eq:u-pos-def}
	u_{k,n}^{[m],+}(y;s_0) &:=
	\begin{dcases} 
		y^{-\frac k2} \frac{\partial^m}{\partial s^m} \sM^+_{\sgn{n}\frac k2,s+\frac{k-1}2}(4\pi|n|y) \Big|_{s=s_0} & \text{ if } s_0\neq \mfrac{1-k}{2}, \\
		y^{-\frac k2} \frac{\partial^{2m}}{\partial s^{2m}} \sM^+_{\sgn{n}\frac k2,s+\frac{k-1}2}(4\pi|n|y) \Big|_{s=\frac{1-k}{2}} &\text{ if }s_0=\mfrac{1-k}{2}.
	\end{dcases}
\end{align}

%%%%%%%%%%%%%%%%%%%
% NEW THEOREM 3.3
%%%%%%%%%%%%%%%%%%%
\begin{theorem}\label{thm:43a}
{\rm (Shifted polyharmonic Fourier coefficients of unrestricted growth, $n \ne 0$)}
Let  $k \in 2 \ZZ$. Suppose that $n \ne 0$ 
and  that $f_{n}(z) = h_{n}(y) e^{2 \pi i n x}$  is a
shifted polyharmonic  function  for $\Delta_k$ on $\HH$ with eigenvalue $\lambda \in \CC$, i.e.
it satisfies
\begin{equation} \label{eq:polyharmonic-f-n}
(\Delta_k- \lambda)^m f_n(z) =0 \quad \mbox{for all} \quad z=x+iy \in \HH.
\end{equation}
Write $\lambda = s_0(s_0+k-1)$ for some $s_0 \in \CC$ (there are generally two choices for $s_0$).
Then
\begin{equation*}
	h_{n}(y) =
	\sum_{j=0}^{m-1} \left( a_{n,j}^{-} \, u_{k,n}^{[j],-}(y;s_0) + a_{n,j}^{+} \, u_{k,n}^{[j],+}(y;s_0) \right)
\end{equation*}
for some constants $a_{n,j}^{-}, a_{n,j}^{+} \in \CC$.
\end{theorem}

\begin{proof}
We know a priori 
{from \eqref{eq:polyharmonic-f-n}
that  $h_n(y)$ must satisfy a linear differential equation of order $2m$, whose} 
 solutions will form a
a vector space of  $2m$ linearly independent functional solutions. 
By \eqref{eq:u-neg-def}, \eqref{eq:u-pos-def}, \eqref{eq:W-plus-def}, and Corollary~\ref{cor:lin-ind} from Appendix~\ref{sec:AppA}, the set
\[
	\left\{ u_{k,n}^{[j],-}(y;s_0) : 0\leq j\leq m-1 \right\} \cup \left\{ u_{k,n}^{[j],+}(y;s_0) : 0\leq j\leq m-1 \right\}
\]
is linearly independent.
It remains to show that
\begin{equation} \label{eq:Delta-k-u-pm}
	\big(\Delta_k-s_0(s_0-k-1)\big)^m u_{k,n}^{[j],\pm}(y;s_0) e^{2\pi i nx} =0
\end{equation}
for all $m\geq 1$ and all $j\leq m-1$.

We proceed by induction on $m$.
Theorem~\ref{thm:41a} shows that \eqref{eq:Delta-k-u-pm} is true for {the base case} $j=0$, $m=1$.
Suppose that \eqref{eq:Delta-k-u-pm} is true for all $m\leq r$ and all $j\leq m-1$ for some $r\geq 1$.
Then clearly \eqref{eq:Delta-k-u-pm} holds for $m=r+1$ and $j\leq r-1$, so it remains to show that it holds for $m=r+1$ and $j=r$.
For brevity, we write $U^{[j]}:=u_{k,n}^{[j],\pm}(y;s) e^{2\pi i nx}$ (thinking of $s$ as a variable) and write $\lambda=s(s+k-1)$.
Then by \eqref{eq:commutation} we have
\begin{align}
	(\Delta_k-\lambda) U^{[r]} 
	&= (\Delta_k-\lambda) \frac{\partial}{\partial s} U^{[r-1]} \notag \\
	&= \frac{\partial}{\partial s} (\Delta_k-\lambda) U^{[r-1]} + (1-k-2s) U^{[r-1]} = \ldots \notag \\
	&= \frac{\partial^r}{\partial s^r} (\Delta_k-\lambda) U^{[0]} + (1-k-2s) \sum_{i=1}^r U^{[r-i]} \notag \\
	&= (1-k-2s) \sum_{i=1}^r U^{[r-i]}
	\label{eq:U-r}
\end{align}
by Theorem~\ref{thm:41a}.
Applying the operator $(\Delta_k-\lambda)^r$ to both sides of \eqref{eq:U-r} we find that
\[
	(\Delta_k-\lambda)^{r+1} U^{[r]} = (1-k-2s) \sum_{i=1}^r (\Delta_k-\lambda)^r U^{[r-i]} = 0
\]
by the induction hypothesis.
The theorem follows. 
\end{proof}

It remains to treat the constant term case, whose solutions  involve power functions and logarithms.
This is straightforward,   but  note that the value $s= \frac{1-k}{2}$ is exceptional.

%%%%%%%%%%%%%%%%%%%
% NEW THEOREM 3.4
%%%%%%%%%%%%%%%%%%%

\begin{theorem}\label{thm:43b}
{\rm (Shifted polyharmonic Fourier coefficients of unrestricted growth,  $n=0$)}
Let  $k \in 2 \ZZ$. Suppose that  
$f_{0}(z) = h_0(y)$  
satisfies
$$
(\Delta_k- \lambda)^m f_0(z) =0 \quad \mbox{for all} \quad z=x+iy \in \HH.
$$
Write $\lambda = s_0(s_0+k-1)$ for some $s_0 \in \CC$ (there are generally two choices for $s_0$).
\begin{enumerate}[\textup(1\textup)]
\item
Suppose that $s_0 \ne \frac{1-k}{2}$ (equivalently,  $\lambda \ne -(\frac{1-k}{2})^2$). Then
\begin{align*}
h_n(y) &= \sum_{j=0}^{m-1} a_{0,j}^{+} \frac{\partial^j}{\partial s^j} y^{s} \Big|_{s=s_0}+
 \sum_{j=0}^{m-1} a_{0,j}^{-} \frac{\partial^j}{\partial s^j} y^{1-k-s} \Big|_{s=s_0} \\
 &= \sum_{j=0}^{m-1} a_{0,j}^{+}  (\log y)^j y^{s_0} + \sum_{j=0}^{m-1} a_{0,j}^{-}  (\log y)^j y^{1-k-s_0}
\end{align*}
for some constants $a_{0,j}^{+}, a_{0,j}^{-} \in \CC$.
\item
Suppose that $s_0 = \frac{1-k}{2}$ (equivalently,   $\lambda = -(\frac{1-k}{2})^2$).
Then
\begin{align*}
h_n(y) &= \sum_{j=0}^{m-1} \frac{\partial^j}{\partial s^j} \left( a_{0,j}^- \, y^{s} + a_{0,j}^+ \, y^{s}\log y \right) \Big|_{s=\frac{1-k}{2}} \\
&= \sum_{j=0}^{m-1} a_{0,2j}^- (\log y)^{2j} y^{\frac{1-k}{2}} + \sum_{j=0}^{m-1} a_{0,2j+1}^- (\log y)^{2j+1} y^{\frac{1-k}{2}}
\end{align*}
for some constants $a_{0,j}^+,a_{0,j}^- \in \CC$.
\end{enumerate}
\end{theorem}

\begin{proof}
The proof mirrors that of Theorem~\ref{thm:43a}, using the commutation relation \eqref{eq:commutation}. \qedhere
\end{proof}

%%%%%%%%%%%%%%%%%%%%%%%%%%%%%%%%%%%%%%%%%%%%%%%%%%%%%%%%%%%%%%%%%%%
%
% Subsection 3.4 HARMONIC Fourier coefficients of moderate growth
%
%%%%%%%%%%%%%%%%%%%%%%%%%%%%%%%%%%%%%%%%%%%%%%%%%%%%%%%%%%%%%%%%%%%%
\subsection{Shifted polyharmonic Fourier coefficients of moderate growth }\label{sec:FourierExp4}

We now characterize the vector spaces of shifted polyharmonic Fourier coefficients of depth $m$ of
moderate growth. These vector spaces have dimension $m$ for all Fourier coefficients with index $n \ne 0$  but
have dimension $2m$ for the constant term coefficient $n=0$.
We immediately obtain the following theorem.

%%%%%%%%%%%%%%%%%%%
% NEW THEOREM 3.5
%%%%%%%%%%%%%%%%%%%
\begin{theorem}\label{thm:45}
{\rm (Shifted polyharmonic Fourier coefficients of moderate growth)}
Let  $k \in 2 \ZZ$. Suppose that $f_{n}(z) = h_{n}(y) e^{2 \pi i n x}$  is a
shifted-polyharmonic  function  for $\Delta_k$ on $\HH$ with eigenvalue $\lambda \in \CC$, i.e.
it satisfies
$$
(\Delta_k- \lambda)^m f_n(z) =0 \quad \mbox{for all} \quad z=x+iy \in \HH.
$$
Suppose also that $f_n(z)$ has at most polynomial growth in $y$ at the cusp.
Then $h_n(y)$ is of the form given in Theorem \ref{thm:43a} with the extra requirement that for $n \ne 0$ 
all coefficients $a_{n,j}^{+} =0$, i.e. no $\sM^+_{\kappa, \mu}(z)$-functions appear in the expansion.
\end{theorem}

\begin{proof} 
This result for $n \ne 0$  is an immediate consequence of 
the asymptotics given in Corollary~\ref{cor:lin-comb-asy} from Appendix~\ref{sec:AppA}. 
The result for $n=0$ follows from Theorem \ref{thm:43b}.
\end{proof}

\subsection{Shifted polyharmonic Fourier expansions}\label{sec45}

We show that Theorem \ref{thm:45}
implies a  Fourier expansion formula valid for  all $m$-harmonic Maass forms
with shifted eigenvalue $\lambda$.   
For $n \ne 0$ define $u_{k,n}^{[m],\pm}(y;s_0)$ by \eqref{eq:u-neg-def} and \eqref{eq:u-pos-def}.
For $n=0$ with $s_0 \ne \frac{1-k}{2}$,
for each $m \ge 0 $ set
\begin{align}
	u_{k,0}^{[m],-}(y;s_0) = 
	\begin{cases}
		(\log y)^m \, y^{1-k-s_0} & \text{ if } s_0\neq \frac{1-k}{2}, \\
		(\log y)^{2m} \, y^{\frac{1-k}{2}} & \text{ if } s_0 = \frac{1-k}{2},
	\end{cases}
	\\
	u_{k,0}^{[m],+}(y;s_0) = 
	\begin{cases}
		(\log y)^m \, y^{s_0} & \text{ if } s_0\neq \frac{1-k}{2}, \\
		(\log y)^{2m+1} \, y^{\frac{1-k}{2}} & \text{ if } s_0 = \frac{1-k}{2}.
	\end{cases}
\end{align}
Then we have the following result, which concerns Fourier expansions for functions
of moderate growth.

%****************************
%
% Theorem 3.6 (3.12)
%
%****************************

\begin{theorem}\label{lem:abstractFourier}
{\rm (Fourier expansion in $V_k^m(\lambda)$)}
Let $f(z) \in V_k^m(\lambda) $ for some $k \in 2\ZZ$.
Let  $m \ge 1$,
and fix an $s_0 \in \CC$ with $\lambda= s_0(s_0+k-1)$. 
Then the Fourier expansion of $f(z)$ exists and  has the form
$$
f(z) = \sum_{j=0}^{m-1} \big(c_{0, j}^{+} u_{k, 0}^{[j],+}(y;s_0)  + c_{0,j}^{-} u_{k, 0}^{[j],-}(y;s_0)  \big)
  + \sum_{\substack{n=-\infty \\ n\neq 0}}^{\infty} \sum_{j=0}^{m-1}  c_{n,j}^{-} u_{k, n}^{[j], -}(y;s_0) e^{ 2\pi i n x},
$$
in which  $c_{n, j}^{\pm}$ are constants. 
This Fourier expansion converges absolutely and uniformly to $f(z)$ on compact subsets of $\HH$.
\end{theorem}
\begin{proof}
See the proof of Theorem 4.3 of \cite{LR15K}.
\end{proof}

\begin{remark}
There is a more explicit version of 
the Fourier expansion for the case $m=1$ and eigenvalue $\lambda=0$,
which is a special case of a Fourier
expansion for  $1$-harmonic Maass forms that appears in the literature.
It contains  incomplete Gamma functions instead of Whittaker 
functions, see  \cite[Lemma 4.4]{LR15K}.
\end{remark}

%%%%%%%%%%%%%%%%%%%%%%%%%%%%%%%%%%%%%%%%%%%%%%%%%%%%%
%%%                                                %%
%%%     4 MAASS   RAISING/LOWERING                 %%
%%%         OPERATORS                              %%
%%%                                                %%
%%%%%%%%%%%%%%%%%%%%%%%%%%%%%%%%%%%%%%%%%%%%%%%%%%%%%

\section{Maass Raising and Lowering Operators}\label{sec:MaassOps}

\subsection{ Properties of the Maass operators }\label{sec:PropMassOps}

Recall that the 
the weight $k$ hyperbolic Laplacian is defined as
\[
	\Delta_k := y^2 \left( \frac{\partial^2}{\partial x^2} + \frac{\partial^2}{\partial y^2} \right) 
- iky\left( \frac{\partial}{\partial x} + i \frac{\partial}{\partial y} \right).
\]
Maass \cite[Chap. 4.1]{Maa83} introduced raising and lowering operators%
\footnote{Maass's original operators are 
$K_{\alpha}= \alpha +(z- \overline{z}) \frac{\partial}{\partial z}$ and $L_{\beta}= - \beta + (z- \overline{z}) \frac{\partial}{\partial \overline{z}},$ 
which differ from the operators $R_k$ and $L_k$ defined in this section. 
The operators $R_k$ and $L_k$ differ as well from those in Bump  \cite[Sec. 2.1]{Bu97}, 
to which they are related by $R_k := \frac{1}{y} (R_k^{B}+ \frac{k}{2})$ and $L_k = -y (L_k^{B} +\frac{k}{2})$, with
superscript $B$ denoting Bump's operators.}
which relate eigenfunctions of $\Delta_k$ with eigenfunctions of $\Delta_{k+2}$ and $\Delta_{k-2}$, respectively.
We follow the convention of \cite[Sec. 2]{BOR08} and define the weight $k$ {\em Maass raising operator} by
$$
R_k := 2i \frac{\partial}{\partial z} + \frac{k}{y} = i\left( \frac{\partial}{\partial x} - i \frac{\partial}{\partial y} \right) + \frac{k}{y},
$$
and the weight $k$ {\em  Maass lowering operator} by
$$
L_k := 2i y^2\frac{\partial}{\partial \bar{z}}  = iy^2\left( \frac{\partial}{\partial x} + i \frac{\partial}{\partial y} \right).
$$
Note that $L_k$ is independent of the weight $k$.
As the following well-known lemma shows, $R_k$ raises the weight by $2$ and $L_k$ lowers the weight by $2$.

\begin{lemma}\label{RLemma1}

For any $\gamma\in \SL(2,\RR)$ we have
\begin{equation}
	R_k \left( f\big|_k \gamma \right) = \left(R_k f\right) \big|_{k+2} \gamma
\end{equation}
and
\begin{equation}
	L_k \left( f\big|_k \gamma \right) = \left(L_k f\right) \big|_{k-2} \gamma.
\end{equation}
\end{lemma}
\begin{proof}
For $\gamma=\smatr abcd$ we write $w=\gamma z=\frac{az+b}{cz+d}$, so that
\[
	\frac{\partial}{\partial z} = \frac{\partial w}{\partial z} \frac{\partial}{\partial w} + \frac{\partial \bar w}{\partial z} \frac{\partial}{\partial \bar w} = (cz+d)^{-2} \frac{\partial}{\partial w}.
\]
We compute
\begin{align*}
	R_k(f\big|_k\gamma) &= -2ikc(cz+d)^{-k-1} f(w) + 2i(cz+d)^{-k-2} \frac{\partial f}{\partial w} + \frac ky (cz+d)^{-k} f(w)\\
	&= (cz+d)^{-k-2} \left[ 2i \frac{\partial f}{\partial w} + \frac{k(cz+d)}y\left(-2icy+cz+d\right) f(w) \right].
\end{align*}
Since $-2icy+cz+d = c\bar z+d$ this becomes
\begin{align*}
	R_k(f\big|_k\gamma) &= (cz+d)^{-k-2} \left[ 2i \frac{\partial f}{\partial w} + \frac{k|cz+d|^2}y f(w) \right] \\
	&= (R_k f) \big|_k \gamma.
\end{align*}
For $L_k$ the proof is similar, and uses the relation
\[
	\frac{\partial}{\partial \bar z} = \frac{\partial w}{\partial \bar z} \frac{\partial}{\partial w} + \frac{\partial \bar w}{\partial \bar z} \frac{\partial}{\partial \bar w} = (c\bar z+d)^{-2} \frac{\partial}{\partial \bar w}. \qedhere
\]
\end{proof}

We recall the following relations (\cite[Sec.2]{BOR08}). 

\begin{lemma}\label{RLemma2}
\begin{enumerate}[\textup(1\textup)]
\item The Laplacian $\Delta_k$ can be expressed in terms of $R_k$ and $L_k$ 
in two ways:
\begin{align*}
-\Delta_k &= L_{k+2} R_k + k, \\
-\Delta_k &=  R_{k-2}{L_k}.
\end{align*}

\item The operators $R_k$ and $L_k$ satisfy the  commutation relations
$$
R_{k-2}{L_k}- L_{k+2} R_k = k.
$$

\item If $f$ is an eigenfunction of $\Delta_k$ satisfying $(\Delta_k-\lambda) f = 0$, then
$R_k f$ and $L_kf$ are also eigenfunctions with shifted eigenvalues:
\begin{gather*}
\big(\Delta_{k+2} - (\lambda+k)\big) R_k f = 0,\\
\big(\Delta_{k-2} - (\lambda+2-k)\big) L_k f = 0.
\end{gather*}
\end{enumerate}
\end{lemma}

\begin{proof}
Relation (1) is a straightforward calculation, and relation (2) follows immediately from (1).
For (3) we obtain using (1) and (2)  
that
$$
\Delta_{k+2}R_k = (-R_k L_{k+2}) R_k = R_k (\Delta_k + k). 
$$
Thus we have the operator identity
\begin{equation}\label{eqn:521}
\big(\Delta_{k+2} -(\lambda+k) \big) R_k = R_k \big(\Delta_k- \lambda \big).
\end{equation}
If $(\Delta_k-\lambda) f = 0$ then \eqref{eqn:521} gives
$$
\big(\Delta_{k+2} -(\lambda +k)\big) R_k f = R_k \big( \Delta_k - \lambda\big) f =0.
$$
A similar calculation gives
the operator identity
\begin{equation}\label{eqn:522}
\big(\Delta_{k-2} -(\lambda+2-k)\big)L_k  = L_k \big( \Delta_k - \lambda\big),
\end{equation}
from which the second part of (3) follows.
\end{proof}

The following lemma generalizes part (3) of Lemma~\ref{RLemma2} to shifted polyharmonic functions.

\begin{lemma}\label{RLemma3}
If $(\Delta_k - \lambda)^m f =0$ then
\begin{gather*}
\big(\Delta_{k+2} - (\lambda +k)\big)^m \, R_k f = 0, \\
\big(\Delta_{k-2} - (\lambda+2-k) \big)^m \, L_k f =0.
\end{gather*}
\end{lemma}

\begin{proof}
For the first equation it suffices to  iterate the 
operator identity  \eqref{eqn:521} to obtain
$$
\big(\Delta_{k+2} -(\lambda +k)\big)^m R_k = \big(\Delta_{k+2} -(\lambda+k)\big)^{m-1}R_k \big(\Delta_k - \lambda\big) = \cdots =R_k\big(\Delta_k - \lambda\big)^m.
$$
The second equation follows by iterating \eqref{eqn:522} similarly.
\end{proof}

\subsection{ Maass operator action on shifted polyharmonic vector spaces}\label{sec:MassOpsShift}

\begin{lemma} \label{lem:Rk-Lk-u}
For $n\neq 0$, let $u_{k,n}^{[m],-}(y;s_0)$ be as in \eqref{eq:u-neg-def}.
Then we have
\begin{align}
	\label{eq:Rk-u}
	R_k \left(u_{k,n}^{[m],-}(y;s_0) e^{2\pi i n x} \right) &= u_{k+2,n}^{[m],-}(y;s_0-1) e^{2\pi i n x} \times
	\begin{cases}
		-1 & \text{ if } n>0, \\
		(s_0+k)(1-s_0) & \text{ if } n<0,
	\end{cases}
	\\
	\label{eq:Lk-u}
	L_k \left(u_{k,n}^{[m],-}(y;s_0) e^{2\pi i n x} \right) &= u_{k-2,n}^{[m],-}(y;s_0+1) e^{2\pi i n x} \times
	\begin{cases}
		s_0(s_0+k-1) & \text{ if } n>0, \\
		1 & \text{ if } n<0.
	\end{cases}
\end{align}
\end{lemma}

\begin{proof}
Since $\frac{\partial}{\partial z}$ and $\frac{\partial}{\partial \bar z}$ commute with $\frac{\partial}{\partial s}$, it suffices to prove the lemma for the case $m=0$.
To simplify notation, we let $\epp=\sgn{n}$ and we write $\mu=s_0+\frac{k-1}{2}$ and $w=4\pi|n|y$.

We begin with the $R_k$ formula.
We have
\[
	\frac{\partial}{\partial z} = \frac 12\left(\frac{\partial}{\partial x} - i\frac{\partial}{\partial y}\right) = -\frac i2 \frac{\partial w}{\partial y}\frac{\partial}{\partial w} = -2\pi i|n| \frac{\partial}{\partial w}.
\]
Hence
\begin{align}
	R_k \left(u_{k,n}^{[m],-}(y;s_0) e^{2\pi i n x} \right) 
	&= (4\pi|n|)^{\frac k2} \left(2i\mfrac{\partial}{\partial z} + \mfrac ky\right) \left( w^{-\frac k2} e^{\frac{\epp w}{2}} W_{\frac{\epp k}{2},\mu}(w) e^{2\pi i n z} \right) \notag \\
	&= (4\pi|n|)^{\frac k2+1} e^{2\pi i n z} \left[ \left(\mfrac kw - \epp + \mfrac{\partial}{\partial w}\right) w^{-\frac k2}e^{\frac{\epp w}{2}} W_{\frac{\epp k}{2},\mu}(w) \right]. \label{eq:Rk-u-w}
\end{align}
If $\epp=+1$ then \cite[(13.15.23) and (13.15.11)]{NIST} give
\begin{equation*} \label{eq:Rk-u-plus}
	\left(\mfrac kw - 1 + \mfrac{\partial}{\partial w}\right) w^{-\frac k2}e^{\frac{w}{2}} W_{\frac{k}{2},\mu}(w) = -w^{-\frac{k+2}{2}}e^{\frac w2} W_{\frac{k+2}{2},\mu}(w),
\end{equation*}
while if $\epp=-1$ then \cite[(13.15.26) and (13.15.11)]{NIST} give
\begin{equation} \label{eq:Rk-u-minus}
	\left(\mfrac kw + 1 + \mfrac{\partial}{\partial w}\right) w^{-\frac k2}e^{-\frac{w}{2}} W_{-\frac{k}{2},\mu}(w) = \left(\mfrac 12 + \mu+\mfrac k2\right) \left(\mfrac 12 - \mu+\mfrac k2\right) w^{-\frac {k+2}2} e^{-\frac w2} W_{-\frac{k+2}{2},\mu}(w).
\end{equation}
Equation \eqref{eq:Rk-u} follows from \eqref{eq:Rk-u-w}, \eqref{eq:Rk-u-plus}, and \eqref{eq:Rk-u-minus} after replacing $w$ by $4\pi|n|y$ and writing $\mu=s_0-1+\frac{k+2-1}{2}$.

The $L_k$ formula is similar.
Using the fact that $\frac{\partial}{\partial \bar z} = 2\pi i|n| \frac{\partial}{\partial w}$ we obtain
\begin{align*}
L_k \left(u_{k,n}^{[m],-}(y;s_0) e^{2\pi i n x} \right) 
&= -(4\pi|n|)^{\frac k2+1} e^{2\pi i n z}y^2 \mfrac{\partial}{\partial w} \left( w^{-\frac k2} e^{\frac{\epp w}{2}} W_{\frac{\epp k}{2},\mu}(w) \right) \\
&= (4\pi|n|)^{\frac k2+1} e^{2\pi i n z}y^2 e^{\frac{\epp w}{2}} w^{-\frac k2-1} W_{\epp\frac {k-2}2,\mu}(w) \\
& \hspace{1.5in} \times
\begin{cases}
	-\left(\mfrac 12 + \mu-\mfrac k2\right) \left(\mfrac 12 - \mu-\mfrac k2\right) & \text{ if }\epp=1, \\
	1 & \text{ if }\epp=-1,
\end{cases}
\end{align*}
where we used \cite[(13.15.23) and (13.15.26)]{NIST} in the last line.
Equation \eqref{eq:Rk-u} follows after replacing $w$ by $4\pi|n|y$ and writing $\mu=s_0+1+\frac{k-2-1}{2}$.
\end{proof}

{Lemmas~\ref{RLemma1} and \ref{RLemma3}  show that the Maass operators $R_k$ and $L_k$ preserve modularity and shift eigenvalues, and 
Lemma  \ref{lem:Rk-Lk-u} implies that they preserve moderate growth of the individual Fourier coefficients.}
These facts leads to the following proposition (illustrated schematically in Figure~\ref{fig:Rk-Lk}).

\begin{figure}[h]
\[
	\boxed{
	\xymatrixcolsep{5pc}
	\xymatrix{
		\ V_{k-2}^m(\lambda+2-k) \ \ar@<.3em>[r]^-{R_{k-2}} & \ar@<.3em>[l]^-{L_k} \ V_k^m(\lambda) \ \ar@<.3em>[r]^-{R_k} & \ar@<.3em>[l]^-{L_{k+2}} \  V_{k+2}^m(\lambda+k) \
	}}
\]
\caption{Action of $R_k$ and $L_k$ on the space $V_k^m(\lambda)$ (Proposition \ref{RLemma4}).}
\label{fig:Rk-Lk}
\end{figure}

\begin{proposition}\label{RLemma4}
 There holds
$$R_k \big(V_k^m(\lambda)\big) \subset V_{k+2}^m(\lambda +k)
$$
and
$$
L_k\big( V_k^m(\lambda)\big) \subset V_{k-2}^m(\lambda + 2-k).
$$
\end{proposition}

\begin{proof}
By the comments above, it suffices to show that if $f\in V_k^m(\lambda)$ then $R_k f$ and $L_k f$ have moderate growth as $y\to\infty$.
We first obtain an estimate {for the size of} the Fourier coefficients $c_{n,j}^-$ of $f$ for $n\neq 0$ (we can safely ignore the index $n=0$ terms since they clearly contribute at most polynomial growth).

Fix $y_0\geq 1$.
By assumption $|f(z)|\ll y^\alpha$ for some $\alpha$ uniformly for $x\in [0,1]$ and $y\geq y_0$.
Thus for each $n\neq 0$ and every $j\leq m$ we have
\[
	c_{n,j}^- \, u_{k,n}^{[j],-}(y;s_0) = \int_0^1 f(z) e^{2\pi i n x} \, dx \ll y^\alpha.
\]
Using the asymptotic formula from Proposition~\ref{prop:whit-deriv-asy} we obtain
\[
	c_{n,j}^{-} \ll y^\beta e^{2\pi |n|y}
\]
as $|n|y\to\infty$,
where the exponent $\beta$ and the implied constant are allowed to depend on $m$ (but not on $n$ or $j$).
Setting $y=\frac{\log|n|}{|n|}$ we find, for some $A\in \RR$, that
\begin{equation} \label{eq:cnj-bound}
	c_{n,j}^- \ll |n|^A.
\end{equation}

Hence, for any fixed $k'\in 2\Z$, $s_0'\in\C$, any $j\leq m$, and any constants $a,b\in\C$, we have the estimate
\begin{align*}
 	a\sum_{n>0} c_{n,j}^- u_{k',n}^{[j],-}(y;s_0') e^{2\pi i n x} + b \sum_{n<0} c_{n,j}^- u_{k',n}^{[j],-}(y;s_0') e^{2\pi i n x} \ll y^B \sum_{n\neq 0} |n|^A e^{-2\pi |n|y_0} \ll y^B
 \end{align*} 
for some $B\in \RR$, as $y\to\infty$.
It follows that both $R_k f$ and $L_k f$ satisfy the moderate growth condition, and this completes the proof.
\end{proof}

We conclude this section by establishing in Proposition~\ref{prop:RkLk-iso} below that, under certain conditions, the maps $R_k$ and $L_k$ are isomorphisms.
First, we prove the following lemma.

%%%%%%%%%%%%%%%%%%%%
% Lemma 4.6
%%%%%%%%%%%%%%%%%%%
\begin{lemma} \label{lem:LkRk-iso}
\begin{enumerate}
\item If $\lambda+k\neq 0$ then the map
\begin{equation*}
	L_{k+2}R_k : V_k^m(\lambda) \to V_k^m(\lambda)
\end{equation*}
is an isomorphism.

\item If $\lambda\neq 0$ then the map
\begin{equation*}
	R_{k-2}L_k : V_k^m(\lambda) \to V_k^m(\lambda)
\end{equation*}
is an isomorphism.
\end{enumerate}
\end{lemma}

\begin{proof}
We prove statement (1); the proof of (2) is analogous.

Suppose that $\lambda+k\neq 0$.
By Lemma~\ref{RLemma2}(1) we have the relation
\begin{equation*}
	L_{k+2}R_k = - \Delta_k - k.
\end{equation*}
It follows that, for $f\in V_k^1(\lambda)$, we have
\[
	L_{k+2}R_k f = -(\lambda+k)f.
\]
Thus $L_{k+2}R_k:V_k^1(\lambda)\to V_k^1(\lambda)$ is surjective.
We proceed by induction. 
Suppose that $m\geq 1$ and that $L_{k+2}R_k: V_k^{m-1}(\lambda)\to V_k^{m-1}(\lambda)$ is surjective.
If $f\in V_k^m(\lambda)$, then $(\Delta_k-\lambda)f\in V_k^{m-1}(\lambda)$.
So by the induction hypothesis, $(\Delta_k-\lambda)f=L_{k+2}R_k g$ for some $g\in V_k^{m-1}(\lambda)$.
We compute
\[
	L_{k+2}R_k (f+g) = -(\Delta_k+k)f+(\Delta_k-\lambda)f = -(k+\lambda) f,
\]
hence $L_{k+2}R_k:V_k^m(\lambda)\to V_k^m(\lambda)$ is surjective.
It follows that $L_{k+2}R_k$ is an isomorphism.
\end{proof}

\begin{proposition} \label{prop:RkLk-iso}
\begin{enumerate}
	\item The map $R_k: V_k^m(\lambda)\to V_{k+2}^m(\lambda+k)$ is an isomorphism when $\lambda+k\neq 0$.
	\item The map $L_k: V_k^m(\lambda)\to V_{k-2}^m(\lambda+2-k)$ is an isomorphism when $\lambda\neq 0$.
\end{enumerate}
\end{proposition}

\begin{proof}
(1) Suppose that $\lambda+k\neq 0$.
Applying Lemma~\ref{lem:LkRk-iso}(1) to $V_k^m(\lambda)$, we find that
\[
	\dim V_{k+2}^m(\lambda+k) \geq \dim V_k^m(\lambda).
\]
Similarly, Lemma~\ref{lem:LkRk-iso}(2) applied to $V_{k+2}^m(\lambda+k)$ gives
\[
	\dim V_k^m(\lambda) \geq \dim V_{k+2}^m(\lambda+k).
\]
It follows that $R_k$ is an isomorphism.

(2) Suppose that $\lambda\neq 0$.
Arguing as in (1), we apply Lemma~\ref{lem:LkRk-iso}(2) to $V_k^m(\lambda)$ and Lemma~\ref{lem:LkRk-iso}(1) to $V_{k-2}^m(\lambda+2-k)$ to conclude that $L_k$ is an isomorphism.
\end{proof}

%%%%%%%%%%%%%%%%%%%%%%%%%%%%%%%%%%%%%%%%%%%%%%%%%%%
%%                                               %%
%%      Section  6 XI OPERATOR                   %%
%%                                               %%
%%%%%%%%%%%%%%%%%%%%%%%%%%%%%%%%%%%%%%%%%%%%%%%%%%%

\section{The $\xi$-Operator and its Properties} \label{sec:newsec5}

\subsection{ Properties of the differential operators $\xi_k$}\label{sec:DifferentialOperators}

Bruinier and Funke \cite[Proposition 3.2]{BF04} introduce the operator
$$
\xi_k := 2i y^k \overline{\frac{\partial}{\partial \bar{z}}}.
$$
This operator is essentially ``half'' of a Laplacian; that is, we have (by a straightforward computation) the relation
\begin{equation} \label{eq:deltak-xik}
	\Delta_k = \xi_{2-k} \xi_k.
\end{equation}
The operator $\xi_k$ is
related to the Maass operator $L_k$ of Section~\ref{sec:MaassOps} by
\begin{equation} \label{eq:xik-Lk}
	\xi_k = -y^{k-2} \overline{L_k},
\end{equation}
and from this it inherits many of its important properties,
as the following lemmas show.

\begin{lemma}\label{lem:xi_commute}
Let $f : \H \to \CC$ be any $C^{1}$-function.
Then for $k\in \Z$ we have
$$
\xi_k  \left(  f\big|_k \gamma  \right)  =  (\xi_k f) \big|_{2-k} \gamma \qquad \text{ for all }\gamma\in \SL(2,\RR).
$$
\end{lemma}

\begin{proof}
From Lemma~\ref{RLemma1} and \eqref{eq:xik-Lk} we have
\begin{align*}
	\xi_k \left(f\big|_k \gamma\right) &= -y^{k-2} \overline{L_k \left(f\big|_k \gamma\right)} = -y^{k-2} \overline{(L_k f)\big|_{k-2}\gamma} \\
	&= y^{k-2} \, \im{\gamma z}^{2-k} \overline{(cz+d)}^{2-k} (\xi_k f)(\gamma z) = (\xi_k f) \big|_k \gamma. \qedhere
\end{align*}
\end{proof}

Lemma \ref{lem:xi_commute} implies that if $f$ is a weight $k$ 
(holomorphic or non-holomorphic) modular form for a discrete subgroup $\Gamma$
of $\SL(2, \ZZ)$, with no growth conditions imposed on any cusp, then $\xi_k f$ is  a
weight $2-k$  modular form for $\Gamma$, again imposing no growth condition at any cusp.

\begin{lemma}\label{lem:EigShift}
Suppose that $f:\H\to\C$ satisfies  $(\Delta_k - \lambda)^m f(z) =0$.
Then we have 
\[
	(\Delta_{2-k} -\overline{\lambda} )^m (\xi_k f) =0.
\]
\end{lemma}

\begin{proof} 
By \eqref{eq:deltak-xik} and the fact that $\xi_k\lambda=\bar\lambda\xi_k$ we have
\begin{equation*}
	(\Delta_{2-k}-\bar\lambda) \xi_k = (\xi_k \xi_{2-k})\xi_k - \bar\lambda \xi_k = \xi_k (\xi_{2-k}\xi_k) - \xi_k \lambda  = \xi_k(\Delta_k - \lambda).
\end{equation*}
Iterating this, we find that
\begin{equation*}
	(\Delta_{2-k} -\overline{\lambda} )^m (\xi_k f) = \xi_k (\Delta_k - \lambda)^mf = 0. \qedhere
\end{equation*}
\end{proof}

\begin{lemma}\label{le54}
Let $u_{k,n}^{[m],-}(y;s_0)$ be as in \eqref{eq:u-neg-def}.
Then
\begin{equation}
	\xi_k \left( u_{k,n}^{[m],-}(y;s_0) e^{2\pi i n x} \right) = u_{2-k,-n}^{[m],-}(y;\overline{s_0}) e^{-2\pi i n x} \times 
	\begin{cases}
		\overline{s_0}(1-k-\overline{s_0}) & \text{ if }n>0, \\
		-1 & \text{ if }n<0.
	\end{cases}
\end{equation}
\end{lemma}

\begin{proof}
This follows immediately from Lemma~\ref{lem:Rk-Lk-u} and \eqref{eq:xik-Lk}, together with the relations
\[
	W_{\kappa,\mu}(y) = W_{\kappa,-\mu}(y)
\]
and (for $\kappa,y\in\RR$)
\[
	\overline{W_{\kappa,\mu}(y)} = W_{\kappa,\bar\mu}(y).
\]
The latter relation follows from the integral representation \cite[(13.16.5)]{NIST} when $\re{\mu}+\frac 12>\re\kappa$ and by analytic continuation otherwise.
\end{proof}

\subsection{Action of $\xi_k$ on non-holomorphic Eisenstein series}\label{sec:Xi-action2}

We compute the action of the $\xi_k$-operator on 
the completed and doubly-completed non-holomorphic Eisenstein series.

\begin{proposition}\label{pr56}
Let $k \in 2 \ZZ$. Then
\[
\xi_k \htE_k(z, s) = 
	\begin{cases}
	\htE_{2-k}(z, -\overline{s}) & \mbox{if} \quad k \le 0,\\
	\overline{s}(\overline{s} + k -1 ) \htE_{2-k}(z, -\overline{s}) & \mbox{if} \quad k \ge 2.
	\end{cases}
\]
In addition
\[
\xi_k \dhtE_k(z, s) = 
	\begin{cases}
	\dhtE_{2-k}(z, -\overline{s}) & \mbox{if} \quad k \le 0,\\
	\overline{s}(\overline{s} + k -1 ) \dhtE_{2-k}(z, -\overline{s}) & \mbox{if} \quad k \ge 2.
	\end{cases}
\]
\end{proposition}

Proposition~\ref{pr56} is proved in \cite[Sect. 7]{LR15K}, which  applies $\xi_k$ directly to the
series expansion \eqref{Eis-k}. This proof is  initially justified  in
the half-plane  $\re{s}> 1- \frac{k}{2}$, then requires analytic continuation in the $s$-variable
to hold in general.  
We outline in Appendix~\ref{sec:AppC}
a second proof of Proposition~\ref{pr56} which is 
based on term-by-term calculation   of
the Fourier series, using Lemma~\ref{le54}. 
This alternate proof uses Whittaker function identities
and is longer, but  has the merit of working directly for all $s \in \CC$ since
the Fourier series converge absolutely and uniformly on compact
subsets (avoiding the poles for the case $k=0$).
Appendix \ref{sec:AppC} gives details of the calculation only for the constant term of the Fourier series.

The following proposition shows that the action of the $\xi_k$ operator on shifted polyharmonic vector spaces preserves the moderate growth condition.

\begin{proposition}\label{lem42}\label{lem:subspaceXi}
For every $m\geq 1$ we have
\[
\xi_k( V_k^m(\lambda)) \subseteq  V_{2-k}^m(\overline{\lambda}).
\]
If $\lambda\neq 0$, this map is an isomorphism.
\end{proposition}

\begin{proof} 
This result  follows from Lemmas~\ref{lem:xi_commute} and \ref{lem:EigShift}, and using the relation $\xi_k= y- y^{k-2}\overline{L_k}$, 
Proposition~\ref{RLemma4} (giving moderate growth), Proposition~\ref{prop:RkLk-iso} (2), and equation \eqref{eq:xik-Lk}.
\end{proof}

\begin{remark}\label{rem69}
On the space $V_k^m(\lambda)$
the operator  $\Delta_k = \xi_{2-k}\xi_k: V_{k}^{m} (\lambda) \to V_{k}^{m}(\lambda)$ 
has minimal polynomial dividing $(T- \lambda)^m$
and is invertible when $\lambda \ne 0$.
For such $\lambda$  we obtain  
a structure of towers for the $\Delta_k$-action, and
a ladder   in which the maps $\xi_k$ and $\xi_{2-k}$
together comprise the rungs.
Figure~\ref{fig:tower-ladder} depicts the tower and ladder structure of these maps.

For $\lambda=0$ the  action of  $\Delta_k$ on $V_k^m(0)$
is nilpotent,  yielding the  tower and ramp structure exhibited in \cite[Tables 1 and 2]{LR15K}.
 In the $\lambda=0$  case there always exist
weight $k$ real-analytic modular forms $f$ that do not have moderate growth at the cusp which
nevertheless have  the property that
$\xi_k f$ has moderate growth at the cusp. The simplest examples are the weakly holomorphic modular forms in $M_k^{!} \setminus M_k$
which have linear exponential growth at the cusp but
 are annihilated by $\xi_k$ since they are holomorphic functions. See \cite[Sect. 6]{LR15K}.
\end{remark}

%%%%%%%%%%%%%%%%%%%%%%%%%%%%%%%%%%%%%%%%%%%%%%
%%                                          %%
%%       SECTION 6:   Taylor coefficients   %%
%%                                          %%
%%%%%%%%%%%%%%%%%%%%%%%%%%%%%%%%%%%%%%%%%%%%%%

\section{Shifted Polyharmonic Maass Forms from  Eisenstein Taylor Coefficients}\label{sec:TS}

We show that the Taylor series coefficients of the
doubly completed Eisenstein series $\dhtE_k(z, s)$
in the $s$-variable  at a point $s=s_0 \in \CC$ define shifted polyharmonic Maass forms
of eigenvalue $\lambda=s_0(s_0+k-1)$.
We consider doubly-completed Eisenstein series 
rather than the singly completed series $\htE_k(z, s)$ because they
are entire functions of $s$ for all $k \in 2\ZZ$; the series  $\htE_0(z, s)$
has simple poles at $s=0, 1$.

\subsection{Taylor series expansions for weight $k$ non-holomorphic Eisenstein series}\label{sec:new71}

The doubly-completed Eisenstein series has a Taylor series
expansion in the $s$-variable around any point $s_0 \in \CC$ given by
\[
\dhtE_k(z, s) = \sum_{n=0}^{\infty} \frac{1}{n!}\dhtE_k^{[n]}(z; s_0) (s- s_0)^n,
\]
in which the Taylor coefficients
\begin{equation}\label{Taylor-coeff0}
\dhtE_k^{[n]}(z; s_0) := \frac{\partial^n}{\partial s^n} \dhtE_k(z, s) \big|_{s=s_0}
\end{equation}
are viewed as functions 
of $z \in \HH$
 to $\CC$.

\begin{theorem}\label{thm:71}
{\rm (Taylor Series Recursion)}
Fix $s_0 \in \CC$ and set $\lambda = s_0(s_0+k-1)$.

\begin{enumerate}
\item
The functions $\dhtE_k^{[n]}(z; s_0)$ 
given by \eqref{Taylor-coeff0}
obey a recursion
\begin{equation}\label{new-recursion}
(\Delta_k - \lambda) \dhtE_k^{[n]}(z; s_0)=
n(2s_0+k-1)\dhtE_k^{[n-1]}(z; s_0) +  n(n-1)\dhtE_k^{[n-2]}(z; s_0),
\end{equation}
in which any terms on the right are omitted whenever 
their superscript $[m]$ has $m <0$.

\item
One has 
$\dhtE_k^{[n]}(z; s_0) \in V_{k}^{n+1}(\lambda)$ for each $n \ge 0$.
That is, $\dhtE_k^{[n]}(z; s_0)$ is a  shifted polyharmonic function
for $\Delta_k$ with (shifted) harmonic depth  at most $n+1$,
with eigenvalue $\lambda$ and moderate growth at the cusp.
\end{enumerate}
\end{theorem}

\begin{proof}
(1)
The  $n=0$ case of \eqref{new-recursion} asserts that
\[
	\Delta_k \dhtE_{k}^{[0]}(z; s_0) = s_0(s_0+k-1) \dhtE_k^{[0]}(z; s_0),
\]
which follows from  Theorem \ref{th38} (3).
For  $n=1$ we  have
\begin{align*}
\Delta_k \dhtE_k^{[1]}(z; s_0)  &= \Delta_k \left(\frac{\partial}{\partial s} \dhtE_k(z, s) \big|_{s=s_0}\right)= 
\frac{\partial}{\partial s} \left[ \Delta_k \dhtE_k(z, s) \right]_{s=s_0}  \\
&= \frac{\partial}{\partial s} \left[ s (s+k-1) \dhtE_k(z, s) \right]_{s=s_0}  \\
&=  s_0(s_0+k-1) \dhtE_k^{[1]}(z; s_0) + (2s_0+k-1) \dhtE_0^{[0]}(z; s_0),
\end{align*}
as required.
The cases $n \ge 2$ are proved by induction on $n$. 
For the induction step, assuming \eqref{new-recursion} holds for $n-1$ and $n-2$, we observe
that
\begin{align*}
\Delta_k \dhtE_k^{[n]}(z; s_0) &=
\frac{\partial}{\partial s}\left[ \Delta_k \dhtE_k^{[n-1]}(z; s) \right]_{s=s_0} \\
&= \frac{\partial}{\partial s} \left[ s (s+k-1) \dhtE_k^{[n-1]}(z; s)+ (n-1) (2s+k-1) \dhtE_k^{[n-2]}(z; s) \right. \\
& \hspace{2.9in}\left. + (n-1)(n-2) \dhtE_k^{[n-3]}(z; s)\right]_{s=s_0}\\
&= \lambda \dhtE_k^{[n]}(z; s_0) + 
n (2s_0+k-1)\dhtE_k^{[n-1]}(z; s_0) +
 n(n-1) \dhtE_k^{[n-2]}(z; s_0),
\end{align*}
which verifies \eqref{new-recursion} for $n$.

(2)
The Taylor coefficients $\dhtE_k^{[n]}(z; s_0)$ 
 inherit the property of transforming as weight $k$ Maass forms
from  $\dhtE_k(z, s)$. 
The recursion \eqref{new-recursion} for $n=0$ states that $\dhtE_k^{[0]}(z; s_0)$ is
shifted polyharmonic with eigenvalue $\lambda$ of  shifted harmonic depth at most $1$. 
By induction on $n$ this recursion establishes 
that $\dhtE_k^{[n]}(z; s_0)$ is shifted polyharmonic of  depth at most $n+1$.
Finally, by Proposition~\ref{prop:FourierExpansionArbitrary} and \eqref{Taylor-coeff0}, the functions 
$\dhtE_k^{[n]}(z; s_0)$ are of moderate growth at the cusp since, aside from the constant term, their Fourier expansions only involve the functions $u_{k,n}^{[m],-}(y;s_0)$, which decay exponentially as $y\to\infty$.
We conclude that  $\dhtE^{[n]}(z; s_0) $ belongs to $V_k^{n+1}(\lambda)$.
\end{proof}

%%%%%%%%%
% Remark 6.2
%%%%%%%%
\begin{rem}\label{rem:72}
In the special case  of the central point $s_0= \frac{1-k}{2}$ of the
functional equation,
where  $\lambda= -(\frac{1-k}{2})^2$, the recursion \eqref{new-recursion}
degenerates to 
\begin{equation}\label{central-pt}
\left(\Delta_k + \left(\tfrac{1-k}{2}\right)^2\right) \dhtE_k^{[n]}\left(z; \tfrac{1-k}{2}\right) = n(n-1) \dhtE_k^{[n-2]} \left(z; \tfrac{1-k}{2}\right).
\end{equation}
The functional equation 
$\dhtE_k(z; s) = \dhtE_k(z; 1-k-s)$ implies that  $\dhtE_k(z; \frac{1-k}{2} +s_1)$ is an even function of $s_1$,
so its odd-indexed Taylor coefficients at $s_1= 0$ vanish identically.  
\end{rem}

\subsection{Shifted polyharmonic Eisenstein series vector spaces}\label{sec:new72}

We define vector spaces of shifted polyharmonic Maass forms spanned
by   Taylor coefficients of non-holomorphic Eisenstein series.

\begin{definition}\label{defn:73}
The  {\em shifted $m$-harmonic Eisenstein space} $E_k^m(\lambda)$
is  the vector space generated by all the Taylor coefficient functions
$\dhtE_k^{[n]}(z; s_0)$ for those $n\geq 0$ such that $\dhtE_k^{[n]}(z; s_0) \in V_k^{m}(\lambda)$.
\end{definition}

As the following theorem shows, these spaces are  finite-dimensional and are generated by all the Taylor coefficient functions
up to some depth $n_0$ depending on $m$ and $\lambda$.

%%%%%%%%%%%%%%%%%%%%%
%  Theorem 6.4
%%%%%%%%%%%%%%%%%%%%%
\begin{theorem}\label{thm:74}
Let $k \in 2\ZZ.$  Then for all  $\lambda \in \CC$ the shifted $m$-harmonic Eisenstein
space $E_k^m(\lambda)$ has dimension $m$. In addition, letting $\lambda= s_0(s_0+k-1)$,  we have
\begin{enumerate}
\item If $\lambda\notin \{\frac k2(1-\frac k2), -(\frac {1-k}2)^2\}$ then
there are two choices for $s_0$. For each choice we have $\dhtE_k^{[0]}(z; s_0) \not\equiv 0$,
and a  basis of $E_k^m(\lambda)$ is given by
\[
	\left\{ \dhtE_k^{[n]}(z; s_0): \, 0 \leq n \leq m-1\right\}.
\]
\item If $\lambda = \frac k2(1-\frac k2)$, so that $s_0=-\frac k2$ or $s_0=1-\frac k2$, then
for all $k \ne 0$,  we have $\dhtE_k^{[0]}(z; s_0)\equiv 0$,
and a basis of 
$E_k^m(\lambda)$ is given by 
\[
\left\{ \dhtE_k^{[n]}(z; s_0): 1 \leq n \leq m\right\}.
\]
If $\lambda=k=0$, we have $\dhtE_k^{[0]}(z; s_0) \not\equiv 0$, and a 
basis of $E_k^m(\lambda)$ is given by
\[
\left\{ \dhtE_k^{[n]}(z; s_0): \, 0 \leq n \leq m-1\right\}.
\]
\item If $\lambda = -(\frac{1-k}{2})^2$ so that $s_0= \frac{1-k}{2}$ is unique, a
basis of $E_k^m(\lambda)$ is given by
\[
\left\{ \dhtE_{k}^{[2n]}(z; s_0): 0 \leq n \leq m-1\right\}.
\]
 
\end{enumerate}
\end{theorem}

%%%%%%%%%%%%%%%%%%%%%
% Proof of Theorem 6.4
%%%%%%%%%%%%%%%%%%%%%

\begin{proof}
It is easy to check via the functional equation relating $s$ to $1-k-s$
that the two values of $s_0$ with $\lambda= s_0(s_0+k-1)$
correspond to the same spaces $E_k^{m}(\lambda)$ for all $m$, so
we may fix one such value.

(1) Suppose that  $\lambda \notin \{-(\frac{1-k}{2})^2,\frac{k}{2}(1-\frac{k}{2})\}$.
It suffices to  show that
\begin{equation} \label{eq:E-k-induction}
	\dhtE_k^{[n]}(z;s_0) \in V_k^{n+1} (\lambda) \setminus V_k^n(\lambda)
\end{equation}
holds for all $n \ge 0$, taking $V_k^0(\lambda) = \{0\}.$ We proceed by induction on $n \ge 0$.
By Proposition \ref{prop:37}, the Fourier
constant term of $\dhtE_k^{[0]}(z;s_0)$ is nonzero since the $s$-polynomial factors in front of the terms $y^{s_0}$ and $y^{1-k-s_0}$ have no common roots.
Hence $\dhtE_k^{[0]}(z;s_0) \not\equiv 0$, which verifies the $n=0$ case of \eqref{eq:E-k-induction}.
By hypothesis $2s_0+k-1\neq 0$, so
the recursion \eqref{new-recursion}
certifies the induction step.

(2) Suppose $\lambda = \frac{k}{2}(1-\frac{k}{2})$ and that $k\neq 0$, so  $s_0=-\frac{k}{2}$ or $1-\frac{k}{2}$ are nonzero. 
We  can see directly from Theorem \ref{th38} that
the completed Eisenstein series $\htE_k(z, s)$ is an entire function,
hence 
\[
	\dhtE_k(z, s_0) = (s_0+\tfrac{k}{2})(s_0+\tfrac{k}{2}-1) \htE_k(z, s_0) = 0,
\]
so $\dhtE_k^{[0]}(z; s_0) \equiv 0$.
Now 
\[
	\dhtE_k^{[1]}(z; s_0) = \frac{\partial}{\partial s} \left[ (s+\tfrac k2)(s+\tfrac k2-1)\htE_k(z,s) \right]_{s=s_0} = \pm \, \htE_k(z,s_0), 
\]
and by Proposition~\ref{prop:37} we see that the Fourier constant term of $\dhtE_k^{[1]}(z; s_0) \not\equiv 0$.
Again $2s_0+k-1\neq 0$, so it  follows by induction on $m \ge 1$ using the recursion  \eqref{new-recursion} 
that $\left\{ \dhtE_k^{[n]}(z; s_0) : 1 \leq n \leq m \right\}$ is a basis of $E_k^m(\lambda)$.

For $\lambda=k=0$ and $s_0\in\{0, 1\}$, Theorem \ref{th38} shows $\htE_k(z;  s)$
has simple poles at $s_0$, coming from the Fourier constant term, so $\dhtE_k^{[0]}(z; s_0) \not\equiv 0$,
The basis result is again proved by induction on $m \ge 1$.
(The case $\lambda=0$ is  treated in detail in Theorems~1.1 and 2.2 of \cite{LR15K}.)

(3) Suppose $\lambda =- (\frac{1-k}{2})^2$, so that $s_0= \frac{1-k}{2}$.
By Proposition~\ref{prop:37} we have the constant term of $\dhtE_k^{[0]}(z; s_0) \not\equiv 0$, and by Theorem~\ref{thm:74}
we have $\dhtE_k^{[0]}(z; s_0) \in V_k^1(\lambda)$. 
By Remark~\ref{rem:72} all Taylor coefficients $\dhtE_k^{[2n+1]}(z; s_0)$ vanish
identically, and  we  have the recursion
$$
\left(\Delta_k + (\tfrac{1-k}{2})^2 \right) \dhtE_k^{[2n]}(z; s_0) = 2n(2n-1) \dhtE_k^{[2n-2]}(z; s_0).
$$
By induction on $n \ge 0$ we conclude that $\dhtE_k^{[2n]}(z; s_0) \in V_k^{n+1}(\lambda) \setminus V_k^{n}(\lambda)$,
the base case $n=0$ having been established.
\end{proof}

\begin{rem}\label{rem:75}
The vector space $E_k^m(\lambda)$ varies continuously as a function of $\lambda$
for all $\lambda \notin \{-(\frac{1-k}{2})^2, \allowbreak \frac k2(1-\frac k2)\}$, in the sense
that the functions $\dhtE_k^{[n]}(z;s_0)$ vary continuously in the parameter $s_0$ (restricting  $z$ to any  compact subset of $\HH$).
It has discontinuous ``jumps'' at the point $\lambda = -(\frac{1-k}{2})^2$ for
all weights and at the point $\lambda = \frac k2(1-\frac k2)$ for all nonzero weights.
\end{rem}

%%%%%%%%%%%%%%%%%%%%%%%%%%%%%%%%%%%%%%%%%%%%%%
%%                                          %%
%%        SECTION 7:   Proofs               %%
%%                                          %%
%%%%%%%%%%%%%%%%%%%%%%%%%%%%%%%%%%%%%%%%%%%%%%

\section{Proof of Theorem~\ref{thm:main1}} \label{sec:Proofs} 

For the reader's convenience, we restate Theorem~\ref{thm:main1} here.

\begin{namedthm*}{Theorem~\ref{thm:main1}}
Fix $k\in 2\Z$. 
For $\lambda\in \CC$ fix $s_0 \in \CC$ such that $\lambda= s_0( s_0 +k -1)$.
\begin{enumerate}
\item The complex vector space $V_k^m(\lambda)$ is finite dimensional, with
\begin{equation}\label{eq:dim-vk-leq}
	\dim V_k^m(\lambda) \leq m + m \dim S_k^1(\lambda),
\end{equation}
where $S_k^1(\lambda)$ is the space of Maass cusp forms of weight $k$ and eigenvalue $\lambda$.

\item This space decomposes as
\[
	V_k^m(\lambda) =E_k^m(\lambda) \oplus S_k^m(\lambda),
\]
in which the Eisenstein series space $E_k^m(\lambda)$ is spanned by certain 
Taylor coefficients of shifted Eisenstein series  and $S_k^m(\lambda)$ is
a recursively defined space of  ``generalized $m$-harmonic Maass cusp forms.''
Both vector spaces $E_k^m(\lambda)$ and
$S_k^m(\lambda)$ are closed under the action of $\Delta_k - \lambda$.
	
\item For all  $\lambda \in \CC$ the space $E_k^m(\lambda)$ has dimension $m$.
\begin{enumerate}[(i)]
 	\item
    For $\lambda \ne -(\frac{1-k}{2})^2$  it  has a basis consisting of 
    the Taylor coefficient functions 
    \[
    	\dhtE^{[j+r]}(z; s_0) := \frac{\partial^{j+r}}{\partial s^{j+r}} \dhtE_k(z; s) \big|_{s= s_0}
		\quad \mbox{for} \quad 0 \le j \le m-1,
	\]
 	where $r$ is minimal such that 
	$\dhtE_k^{[r]}(z; s_0) \not\equiv 0$. Here $r=0$ unless  $\lambda=\frac k2(1-\frac k2)$
	and $k \ne 0$, in which case   $r=1$. 
	
	\item
	For $\lambda = - (\frac{1-k}{2})^2$ and $s_0 =\frac{1-k}{2}$,  a basis is given by  the even-indexed Taylor
	coeffcient functions $\dhtE^{[2j]}(z; s_0)$ 
	for $0 \le j \le m-1$. All odd-indexed functions $\dhtE^{[2j+1]}(z; s_0) \equiv 0$.
\end{enumerate}

\item 
For $m \ge 1$ one has 
\[
	\dim \big( S_k^{m}(\lambda)\big) \le m \dim \big(S_k^1(\lambda) \big).
\]
\end{enumerate}
\end{namedthm*}

\begin{proof}
(1)
We prove the upper bound \eqref{eq:dim-vk-leq} by induction on $m \ge 1$. 
The base case $m=1$ holds since $V_k^1(\lambda) = E_k^1(\lambda) \oplus S_k^1(\lambda)$, where $S_k^1(\lambda)$ is the space of Maass cusp forms 
(those forms having Fourier constant term equal to $0$), and since $\dim E_k^1(\lambda)=1$ by Theorem~\ref{thm:74}.
We verify by induction on $m \ge 1$ the hypothesis
\begin{equation} \label{eq:dim-ind-hyp}
\dim  V_k^{m}(\lambda) - \dim V_{k}^{m-1}(\lambda) \le \dim E_k^{1}(\lambda) + \dim S_k^1(\lambda),
\end{equation}
where we set $V_k^0(\lambda)=\{0\}$.
Since $V_k^{m-1}(\lambda) \subset V_k^{m}(\lambda)$ we may define the quotient space 
$W_k(\lambda) := V_k^{m}(\lambda)/ V_k^{m-1}(\lambda)$.
Since $(\Delta_k - \lambda) : V_k^m(\lambda)  \to V_k^{m-1} (\lambda)$ for all $m\geq 1$, the map
\begin{equation} \label{eq:wk-map}
	(\Delta_k - \lambda) : W_k^m(\lambda) \to W_k^{m-1}(\lambda)
\end{equation}
is well-defined.
We claim this map is injective.
Given $f\in V_k^m(\lambda)$ let $[f]$ denote the coset in $W_k^m(\lambda)$ containing $f$.
Let $f,g\in V_k^m(\lambda)$ and suppose that
\[
	(\Delta_k - \lambda)[f] = (\Delta_k - \lambda)[g] \quad \mbox{in} \quad W_k^{m-1}(\lambda).
\]
Then $(\Delta_k-\lambda)(f-g)\in V_k^{m-2}(\lambda)$.
It follows that 
\[
	(\Delta_k-\lambda)^{m-1}(f-g) = (\Delta_k-\lambda)^{m-2} \big((\Delta_k-\lambda)(f-g)\big) = 0,
\]
whence $f-g\in V_k^{m-1}$, i.e. $[f]=[g]$.
This shows that the map \eqref{eq:wk-map} is injective, from which it follows that
\[
	\dim W_k^m(\lambda) \le \dim W_k^{m-1}(\lambda).
\]
The latter inequality verifies the induction hypothesis \eqref{eq:dim-ind-hyp}.
We conclude that
\[
	\dim V_k^m(\lambda) \leq m + m \dim S_k^1(\lambda).
\]

(2) Given $f\in V_k^m(\lambda)$ let
\[
	CT(f) = \sum_{j=0}^{m-1} \left( c_{0,j}^+(f) u_{k,0}^{[j],+}(y,s_0) + c_{0,j}^-(f)u_{k,0}^{[j],-}(y,s_0) \right)
\]
denote the Fourier constant term of $f$ (see Theorem~\ref{lem:abstractFourier}).
We define a Hermitian scalar product on the vector space $V_k^m(\lambda)$ by
\begin{equation}
	\left\langle F, G \right\rangle_{s_0} := \sum_{j\geq 0} \overline{c_{0,j}^+(F)} c_{0,j}^+(G) + \sum_{j\geq 0} \overline{c_{0,j}^-(F)} c_{0,j}^-(G).
\end{equation}
This sum is finite because $c_{0,j}^+(f)=c_{0,j}^-(f)=0$ for $j$ sufficiently large.

The spaces $E_k^m(\lambda)$ were defined in Section~\ref{sec:new72}.
We start for $m=1$ with the usual decomposition
\begin{equation*}
	V_k^1(\lambda) = E_k^1(\lambda) + S_k^1(\lambda)
\end{equation*}
in which $E_k^1(\lambda)$ is spanned by  a single Eisenstein series, 
$\widetilde{E}_k(z) := \dhtE_k^{[r]}(z;s_0)$, where $r\in \{0,1\}$ (see Theorem~\ref{thm:74}).
{The proof of Theorem ~\ref{thm:74} showed that $\widetilde{E}_k(z)$ has a non-vanishing Fourier constant term.}
The space $S_k^1(\lambda)$ consists of the cusp forms, which have identically zero Fourier constant term.

For the Eisenstein series $\widetilde{E}_k(z)$, we have all $c_{0,j}^\pm(\widetilde{E}_k)=0$ if $j\geq 1$
(since it is annihilated by $\Delta_k -\lambda$) hence for all $G\in V_k^m(\lambda)$,
\begin{equation} \label{eq:<Ek,G>}
	\left\langle \widetilde{E}_k, G \right\rangle_{s_0} = \overline{c_{0,0}^+(\widetilde{E}_k)} c_{0,0}^+(G) + \overline{c_{0,0}^-(\widetilde{E}_k)} c_{0,0}^-(G).
\end{equation}
Here at least one of $c_{0,0}^+(\widetilde{E}_k)$, $c_{0,0}^-(\widetilde{E}_k)$ is nonzero.

We now recursively construct  for $m\geq 2$ a decomposition
\begin{equation*}
	V_k^m(\lambda) = E_k^m(\lambda) + S_k^m(\lambda),
\end{equation*}
using the Hermitian scalar product as follows.
Let $\widetilde{S}_k^m(\lambda)$ denote the space
\[
	\widetilde{S}_k^m(\lambda) := \left\{ G \in V_k^m(\lambda) : (\Delta_k-\lambda) G \in S_k^{m-1}(\lambda) \right\},
\]
and define, for $m\geq 2$, the subspace
\begin{equation*}
	S_k^m(\lambda) := \left\{ G \in \widetilde S_k^m(\lambda) : \left\langle \widetilde E_k, G \right\rangle_{s_0} =0 \right\}.
\end{equation*}
Note that $S_k^1(\lambda)$ satisfies the property above as well.
This definition ensures that $\widetilde E_k(z)\notin S_k^m(\lambda)$, and that
\begin{equation} \label{eq:Delta-lowers-depth}
	(\Delta_k-\lambda) S_k^m(\lambda) \subset S_k^{m-1}(\lambda).
\end{equation}
In the following claims, we prove that the ``generalized cusp form'' space $S_k^m(\lambda)$ has the required properties.

{\em Claim 1}: $S_k^{m-1}(\lambda)\subseteq S_k^m(\lambda)$.

The definition of $S_k^m(\lambda)$ includes all elements of $S_k^{m-1}(\lambda)$ because the orthogonality condition for $\widetilde E_k(z)$ in the product uses only the first two coefficients in the constant term 
(see \eqref{eq:<Ek,G>}). This proves Claim 1.

{\em Claim 2}: $E_k^m(\lambda) \cap S_k^m(\lambda) = \{0\}$.

Suppose, by way of contradiction, that $0\neq G(z) \in E_k^m(\lambda) \cap S_k^m(\lambda)$. 
By \eqref{eq:Delta-lowers-depth} and Claim 1, $S_k^m(\lambda)$ is closed under the action of $\Delta_k-\lambda$.
We first treat the case $\lambda \notin \{-(\frac {1-k}2)^2, \frac k2(1-\frac k2)\}$.
By Theorem~\ref{thm:74}(1) we have
\[
	G(z) = \sum_{j=0}^{m-1} \alpha_j \dhtE_k^{[j]}(z,s_0).
\]
Let $\ell$ denote the largest integer such that $\alpha_\ell \neq 0$.
Now $ G \in S_k^m(\lambda)$ so  $(\Delta_k-\lambda)^{\ell} G \in S_k^m(\lambda)$.
However, using the recursion of Theorem~\ref{thm:71} we have
\[
	(\Delta_k-\lambda)^{\ell}(G(z)) = \alpha_\ell (\ell)! (2s_0+k-1)^{\ell} \dhtE_k^{[0]}(z;s_0).
\]
Since $2s_0+k-1\neq 0$ this is a nonzero scalar multiple of $\widetilde E_k(z)$, contradicting $\widetilde E_k(z)\notin S_k^m(\lambda)$.
The cases $\lambda \in \{-(\frac {1-k}2)^2, \frac k2(1-\frac k2)\}$ are handled by similar arguments using Theorem~\ref{thm:74} (2) and (3).
This proves Claim 2.

{\em Claim 3}: $V_k^m(\lambda) = E_k^m(\lambda) + S_k^m(\lambda)$.

By Claim 2 it suffices to show that
\begin{equation} \label{eq:dim-geq}
	\dim \left(E_k^m(\lambda)\right) + \dim\left(S_k^m(\lambda)\right) \geq \dim \left(V_k^m(\lambda)\right).
\end{equation}
We first show via an inductive argument on $m \ge 1$ that
\begin{equation} \label{eq:dim-geq-til}
	\dim \left(\widetilde S_k^m(\lambda) \right) \geq \dim \left(V_k^m(\lambda)\right) - \dim \left(E_k^m(\lambda)\right) + 1.
\end{equation}
{ This bound holds for the base case $m=1$ since $\widetilde S_k^1(\lambda) = V_k^1(\lambda)$ and $\dim \left(E_k^m(\lambda)\right)=1$.}
Now let  $F\in V_k^m(\lambda)$ so $(\Delta_k-\lambda)F \in V_k^{m-1}(\lambda)$, where $V_k^{m-1}(\lambda) = E_k^{m-1}(\lambda) + S_k^{m-1}(\lambda)$ by the inductive hypothesis.
Now $(\Delta_k-\lambda) E_k^m(\lambda)$ has image $E_k^{m-1}(\lambda)$ and $(\Delta_k-\lambda) \widetilde E_k(z) = 0$.
Therefore there is a subspace $\widetilde E_k^m(\lambda) \subseteq E_k^m(\lambda)$ of codimension at least $1$ whose image under $(\Delta_k-\lambda)$ is the full space $E_k^{m-1}(\lambda)$.
Thus we can find an element $ E \in \widetilde E_k^m(\lambda) $ such that $(\Delta_k-\lambda)(F- E) \in S_k^{m-1}(\lambda)$.
Therefore
\[
	\dim \left(\widetilde S_k^m(\lambda) \right) \geq \dim \left(V_k^m(\lambda)\right) - \dim \left( \widetilde E_k^m(\lambda) \right) \ge
\dim \left(V_k^m(\lambda)\right) - \dim \left(E_k^m(\lambda)\right)	+1,
\]
which proves \eqref{eq:dim-geq-til}.
To prove \eqref{eq:dim-geq} it suffices to show that
\begin{equation} \label{eq:Sk-dim-leq}
	\dim \left(\widetilde S_k^m(\lambda)\right) \leq \dim \left(S_k^m(\lambda)\right) + 1.
\end{equation}
This bound follows from the definition of $S_k^m(\lambda)$, which shows that every element of $\widetilde S_k^m(\lambda)$ differs from an element of $S_k^m(\lambda)$ by a multiple of $\widetilde E_k(z)$.
The inequalities \eqref{eq:dim-geq-til} and \eqref{eq:Sk-dim-leq} together prove \eqref{eq:dim-geq}. This proves Claim 3.

(3) This result summarizes the content of Theorem \ref{thm:74}.

(4) This follows immediately from statements (1) and (3). \qedhere
\end{proof}

\subsection*{Acknowledgments}
Some  of the work of the  second author (J. L.) 
was done at ICERM, where he was a Clay Senior Fellow; he thanks  the Clay Foundation for support.
The authors  thank Mike Kelly for helpful remarks.

\appendix

%%%%%%%%%%%%%%%%%%%%%%%%%%%%%%%%%%%%%%%%%%%%%%%%%%%%
%%                                                %%
%%      Appendix A whittaker asymptotics          %%
%%                                                %%
%%%%%%%%%%%%%%%%%%%%%%%%%%%%%%%%%%%%%%%%%%%%%%%%%%%%

\section{Derivatives of Whittaker functions in the second index}\label{sec:AppA}

We give asymptotic expansions for Whittaker functions and for the
derivatives with respect to $\mu$ of parametrized families of Whittaker functions,
used in Section~\ref{sec:FourierExp4}.
Fix $\delta>0$. 
For $z \in \C$ with $|z|\to \infty$ in $|\arg z|\leq \frac{3\pi}2-\delta$ we have the asymptotic expansion \cite[(13.19.3)]{NIST}
\begin{equation} \label{eq:W-asymp}
	W_{\kappa,\mu}(z) \sim e^{-\frac z2} z^\kappa \sum_{n=0}^\infty \frac{\left(\frac 12+\mu-\kappa\right)_n\left(\frac 12-\mu-\kappa\right)_n}{n!} (-z)^{-n},
\end{equation}
where $(a)_n$ is the Pochhammer symbol
\[
	(a)_n := \frac{\Gamma(a+n)}{\Gamma(a)} = a(a+1)\cdots(a+n-1).
\]
When one of $\frac 12\pm \mu-\kappa$ is a negative integer or zero, the series \eqref{eq:W-asymp} terminates; in that case $W_{\kappa,\mu}(z)$ is equal to its asymptotic expansion.

Differentiating asymptotic expansions (especially with respect to parameters) requires special care.
In the following proposition we show that the asymptotic expansions of the $\mu$-derivatives of $W_{\kappa,\mu}(z)$ are obtained by differentiating \eqref{eq:W-asymp} term-by-term.

%%%%%%%%%%%%%%%%%%%%%
% Proposition A-1
%%%%%%%%%%%%%%%%%%%%%

\begin{proposition} \label{prop:whit-deriv-asy}
Fix $\delta>0$, and for $\mu_0\in \C$ fix a small open neighborhood $U\subset \C$ of $\mu_0$.
For $\mu\in U$ and for $z \in \C$ with $|z|\to \infty$ in $|\arg z|\leq \frac{3\pi}2-\delta$ we have the asymptotic expansion
\begin{equation} \label{eq:whit-deriv-asy}
	\frac{\partial^m}{\partial \mu^m} W_{\kappa,\mu}(z) \sim e^{-\frac z2} z^\kappa \sum_{n=0}^\infty \frac{\partial^m}{\partial \mu^m} \frac{\left(\frac 12+\mu-\kappa\right)_n\left(\frac 12-\mu-\kappa\right)_n}{n!} (-z)^{-n}.
\end{equation}
\end{proposition}

\begin{proof}
We begin with the Mellin-Barnes integral representation \cite[Section 16.4]{WW27} for $W_{\kappa,\mu}(z)$, valid for $|\arg z|\leq \frac{3\pi}{2}-\delta$ and all $\kappa,\mu$ such that neither of $\frac 12\pm\mu+\kappa$ is a positive integer:
\begin{equation} \label{W-mellin-barnes}
	M_{\kappa,\mu}(z) = \frac{e^{-\frac z2}z^\kappa}{2\pi i} \int_{-i\infty}^{i\infty} \frac{\Gamma(s)\Gamma(\frac 12-\mu-\kappa-s)\Gamma(\frac 12+\mu-\kappa-s)}{\Gamma(\frac 12-\mu-\kappa)\Gamma(\frac 12+\mu-\kappa)} z^s \, ds.
\end{equation}
The contour in \eqref{W-mellin-barnes} loops if necessary so that it separates the poles of $\Gamma(s)$ and 
\[
	G_{\kappa,\mu}(s) := \frac{\Gamma\left(\frac 12-\mu-\kappa-s\right)\Gamma\left(\frac 12+\mu-\kappa-s\right)}{\Gamma(\frac 12-\mu-\kappa)\Gamma(\frac 12+\mu-\kappa)}.
\]
Fix a positive integer $N$ such that for each $\mu\in U$ (with neither of $\frac 12\pm\mu+\kappa$ a positive integer), the poles of $G_{\kappa,\mu}(s)$ are to the right of the line $\re{s}=-N-\frac 12$.
Following Section 16.4 of \cite{WW27} we find that
\begin{equation} \label{eq:W-residue+int}
	W_{\kappa,\mu}(z) = e^{-\frac z2}z^{\kappa} \left( \sum_{n=0}^N  \frac{G_{\kappa,\mu}(-n)}{n!}(-z)^{-n} + \frac{1}{2\pi i} \int_{-N-\frac 12-i\infty}^{-N-\frac 12+i\infty} \Gamma(s) G_{\kappa,\mu}(s) z^s\, ds \right).
\end{equation}
As above, the contour loops if necessary to avoid poles of the integrand.
Note that for $\kappa,\mu$ with one of $\frac 12\pm\mu+\kappa$ a positive integer, we have $G_{\kappa,\mu}(s)\equiv 0$, so \eqref{eq:W-residue+int} holds for all $\mu\in U$.

We can differentiate with respect to $\mu$ under the integral sign as long as
\begin{equation} \label{eq:int-deriv-N}
	\int_{-N-\frac 12-i\infty}^{-N-\frac 12+i\infty} \frac{\partial^m}{\partial \mu^m} \Gamma(s) G_{\kappa,\mu}(s) z^s\, ds
\end{equation}
is absolutely convergent.
Let $\psi(s)$ denote the digamma function $\psi(s)=\frac{\Gamma'}{\Gamma}(s)$, and define
\[
	H_{\kappa,\mu}(s) := -\psi\left(\tfrac 12-\mu-\kappa-s\right) + \psi\left(\tfrac 12+\mu-\kappa-s\right) - \psi\left(\tfrac 12-\mu-\kappa\right) + \psi\left(\tfrac 12+\mu-\kappa\right),
\]
so that
\[
	\frac{\partial}{\partial \mu} G_{\kappa,\mu}(s) = G_{\kappa,\mu}(s) H_{\kappa,\mu}(s).
\]
Iterating the latter equation, we find that
\begin{equation} \label{eq:G-deriv-psi}
	\frac{\partial^m}{\partial \mu^m} G_{\kappa,\mu}(s) = G_{\kappa,\mu}(s) \left( H_{\kappa,\mu}(s)^m + R \right),
\end{equation}
where $R$ is a polynomial of degree $m-1$ in $H_{\kappa,\mu}(s)$ and its $\mu$-derivatives.
For large $|s|$ with $|\arg s|\leq \pi-\delta$ we have  (by differentiating the asymptotic expansion \cite[(5.11.2)]{NIST}, see \cite[Section 2.1(ii)]{NIST}) the estimates
\begin{equation} \label{eq:psi-asymp}
	\psi(s) \sim \log |s| \quad \text{ and } \quad \psi^{(j)}(s) \asymp \frac{1}{|s|^j}, \quad j\geq 1.
\end{equation}
Setting $s=-N-\frac 12+it$, it follows from \eqref{eq:G-deriv-psi}, \eqref{eq:psi-asymp}, and \cite[(5.11.9)]{NIST} that
\begin{equation}
	\frac{\partial^m}{\partial \mu^m} \Gamma(s) G_{\kappa,\mu}(s) z^s \ll_{m,\kappa,\mu} |z|^{-N-\frac 12} e^{-\frac{3\pi}2|t|} |t|^{N-2\re{\kappa}} (\log|t|)^m \quad \text{ as }t\to\pm\infty.
\end{equation}
Thus the integral \eqref{eq:int-deriv-N} is absolutely convergent and is $O(|z|^{-N-\frac 12})$.
It follows that for $\mu\in U$, the function $\frac{\partial^m}{\partial \mu^m} W_{\kappa,\mu}(z)$ has asymptotic expansion
\[
	\frac{\partial^m}{\partial \mu^m} W_{\kappa,\mu}(z) \sim e^{-\frac z2} z^\kappa \sum_{n=0}^\infty \frac{\partial^m}{\partial \mu^m} \frac{G_{\mu,\kappa}(-n)}{n!} (-z)^{-n}.
\]
This completes the proof since
\[
	G_{\kappa,\mu}(-n) = \left(\mfrac 12+\mu-\kappa\right)_n\left(\mfrac 12-\mu-\kappa\right)_n. \qedhere
\]
\end{proof}

The following corollary of Proposition~\ref{prop:whit-deriv-asy} is vital for the results of Section~\ref{sec:newsec4}.

%%%%%%%%%%%%%%%%%%%%%
% Corollary A.2
%%%%%%%%%%%%%%%%%%%%%

\begin{corollary} \label{cor:lin-ind}
Fix $\mu_0\in \C$.
If $\mu_0\neq 0$ then for each $m\geq 0$ the set
\begin{equation} \label{eq:ind-set-1}
	\left\{ \frac{\partial^j}{\partial \mu^j} W_{\kappa,\mu}(z) \big|_{\mu=\mu_0} : 0\leq j\leq m \right\} \cup \left\{ \frac{\partial^j}{\partial \mu^j} W_{-\kappa,\mu}(-z) \big|_{\mu=\mu_0} : 0\leq j\leq m \right\}
\end{equation}
is linearly independent.
If $\mu_0=0$ then for each $m\geq 0$ the set
\begin{equation} \label{eq:ind-set-2}
	\left\{ \frac{\partial^{2j}}{\partial \mu^{2j}} W_{\kappa,\mu}(z) \big|_{\mu=0} : 0\leq j\leq m \right\} \cup \left\{ \frac{\partial^{2j}}{\partial \mu^{2j}} W_{-\kappa,\mu}(-z) \big|_{\mu=0} : 0\leq j\leq m \right\}
\end{equation}
is linearly independent.
\end{corollary}

\begin{proof}
Suppose first that $\mu_0\neq 0$.
For each $j\geq 0$ and $\ell=1,2$ let
\begin{equation*}
	F_{2j+\ell}(z) := \frac{\partial^{2j+\ell}}{\partial \mu^{2j+\ell}} \left( \frac{(\frac 12+\mu-\kappa)_{j+1}(\frac 12-\mu-\kappa)_{j+1}}{(j+1)!(-z)^{j+1}} + \frac{(\frac 12+\mu-\kappa)_{j+2}(\frac 12-\mu-\kappa)_{j+2}}{(j+2)!(-z)^{j+2}} \right)\Bigg|_{\mu=\mu_0}.
\end{equation*}
Since $(\frac 12+\mu-\kappa)_n(\frac 12-\mu-\kappa)_n$ is a polynomial in $\mu^2$ of degree $n$, we have
\begin{equation*}
	F_{2j+1}(z) = \frac{(2j+2)!}{(j+1)!}\frac{\mu_0}{z^{j+1}} + \frac{(2j+4)!}{6(j+2)!}\frac{\mu_0^3}{z^{j+2}} + \frac{\alpha_j \mu_0}{z^{j+2}}
\end{equation*}
for some $\alpha_j \in \C[\kappa]$.
Hence
\begin{equation*}
	\mu_0 \, F_{2j+2}(z) - F_{2j+1}(z) = \frac{(2j+4)!}{3(j+2)!}\frac{\mu_0^3}{z^{j+2}},
\end{equation*}
from which it follows that $F_{2j+1}(z)$ and $F_{2j+2}(z)$ are linearly independent (since $\mu_0\neq 0$).

By Proposition~\ref{prop:whit-deriv-asy} and the fact that $\frac{\partial^{2j+1}}{\partial \mu^{2j+1}}(\frac 12+\mu-\kappa)_n(\frac 12-\mu-\kappa)_n = 0$ for $n\leq j$ we have the asymptotic formula
\begin{equation*}
	\frac{\partial^{2j+\ell}}{\partial \mu^{2j+\ell}} W_{\kappa,\mu}(z) \sim e^{-\frac z2} z^\kappa F_{2j+\ell}(z).
\end{equation*}
It follows that the set \eqref{eq:ind-set-1} is linearly independent.

If $\mu_0=0$ then since the right-hand side of \eqref{eq:whit-deriv-asy} is an even function of $\mu$, all of the odd-order derivatives $\frac{\partial^{2j+1}}{\partial \mu^{2j+1}}W_{\kappa,\mu}(z) \big|_{\mu=0}$ are identically zero.
But by Proposition~\ref{prop:whit-deriv-asy} we have the asymptotic formula
\begin{equation*}
	\frac{\partial^{2j}}{\partial \mu^{2j}}W_{\kappa,\mu}(z) \sim \frac{(2j)!}{j!} e^{-\frac z2} z^{\kappa-j},
\end{equation*}
from which it follows that the set \eqref{eq:ind-set-2} is linearly independent.
\end{proof}

The proof of Corollary~\ref{cor:lin-ind} has the following immediate consequence.

%%%%%%%%%%%%%%%%%%%%%
% Corollary A.3
%%%%%%%%%%%%%%%%%%%%%

\begin{corollary} \label{cor:lin-comb-asy}
Suppose that $y>0$.
All $\C$-linear combinations of the functions $\frac{\partial^j}{\partial\mu^j} W_{\kappa,\mu}(y)$ decay exponentially as $y\to \infty$, while all 
{nonzero} $\C$-linear combinations of the functions $\frac{\partial^j}{\partial\mu^j} W_{-\kappa,\mu}(-y)$ grow exponentially as $y\to\infty$.
\end{corollary}

%%%%%%%%%%%%%%%%%%%%%%%%%%%%%%%%%%%%%%%%%%%%%%%%%%%%
%%                                                %%
%%      Appendix B                                %%
%%                                                %%
%%%%%%%%%%%%%%%%%%%%%%%%%%%%%%%%%%%%%%%%%%%%%%%%%%%%

\section{Action of $\xi_k$-operator  on non-holomorphic Eisenstein series}\label{sec:AppC}

This appendix sketches  an alternate proof of  Proposition \ref{pr56} which works directly for all $s \in \CC$.

\begin{proposition}\label{prB1}
Let $k \in 2 \ZZ$. Then
$$
\xi_k \htE_k(z, s) = \begin{cases}
\htE_{2-k}(z, -\overline{s}) & \mbox{if} \quad k \le 0,\\
\overline{s}(\overline{s} + k -1 ) \htE_{2-k}(z, -\overline{s}) & \mbox{if} \quad k \ge 2.
\end{cases}
$$
\end{proposition}

\begin{proof}
We compute the action of $\xi_k$ on the Fourier series of $E_k(z,s)$ term by term,
using the formulas in Lemma \ref{le54}. We assert that under the action of $\xi_k$ the $n$-th
the Fourier term maps to the $-n$-th term of $\htE_{2-k}(z, - \overline{s})$,
multiplied by the appropriate constant ($1$ or $\overline{s}(\overline{s}+k-1)$).
Futhermore, for the constant term the coefficients
of $y^s$ and $y^{1-s-k}$ are interchanged,
again multiplied by the appropriate constant.
Here we supply details proving the assertion for one specific case,  sufficient to uniquely
determine the multiplying constants. 
We write
\[
	C_0(y,s) = CT_k^+(s)y^s + CT_k^-(s)y^{1-s-k}
\]
for the constant term of $\htE_{k}(z,s)$.
By Proposition \ref{prop:FourierExpansionArbitrary} we have
\[
	CT_{k}^{+}(s) =  \frac{\Gamma \left( s+\frac{k}{2} + \frac{\abs{k}}{2} \right) }{\Gamma \left( s+\frac{k}{2} \right) }  \htz(2s+k)
\]
and
\[
CT_{k}^{-}(s) =(-1)^{\frac{k}{2}} \frac{\Gamma \left(  s+ \frac{k}{2} \right)  \Gamma  \left(  s+ \frac{k}{2} + \frac{\abs{k}}{2} \right) }{\Gamma( s+ k) \Gamma(s) } \htz(2-2s-k).
\]
We will show that
\begin{equation} \label{eq:CT-switch}
\xi_k\left(CT_{k}^{+}(s)y^s\right) = 
\begin{cases}
	CT_{2-k}^{-}(-\overline{s})y^{\bar s+k-1} & \text{ if } k \le 0,\\
	\overline{s}(\overline{s} + k -1) CT_{2-k}^{-}(-\overline{s})y^{\bar s+k-1} & \text{ if } k \ge 2.
\end{cases}
\end{equation}
The proof depends on the value of $k$ and  uses identities
for the Gamma function.

First, suppose that $k\leq 0$.
We compute that
\[
\xi_k\left(CT_k^+(s)y^s\right) = y^k CT_k^+(\bar s) \overline{\frac{\partial}{\partial y}y^s} = \frac{\bar s \Gamma(\bar s)}{\Gamma(\bar s+\frac k2)} \hat\zeta(2\bar s+k) y^{\bar s+k-1}.
\]
On the other hand,
\begin{align*}
CT_{2-k}^{-}(- \overline{s}) y^{\bar s+k-1} = (-1)^{1-\frac{k}{2}}\frac{\Gamma (1- \overline{s} - \frac{k}{2})} {\Gamma( -\overline{s})} \htz(2 \overline{s} +k)  y^{\overline{s}+k -1}.
\end{align*}
It remains to show (replacing $\bar s$ by $s$) that
\[
	\frac{s\Gamma(s)}{\Gamma(s+\frac k2)} = (-1)^{1-\frac k2} \frac{\Gamma(1-s-\frac k2)}{\Gamma(-s)}.
\]
Indeed, using $(-z)\Gamma(-z) = \Gamma(1-z)$ and $\Gamma(z) \Gamma(1-z) = \frac{\pi}{\sin \pi z}$ we find that
\[
	\frac{s\Gamma(-s)\Gamma(s)}{\Gamma(s+\frac k2)\Gamma(1-s-\frac k2)} = -\frac{\Gamma(s)\Gamma(1-s)}{\Gamma(s+\frac k2)\Gamma(1-s-\frac k2)} = -\frac{\sin \pi(s+\frac k2)}{\sin\pi s} = (-1)^{1-\frac k2},
\]
as desired.

Now suppose that $k\geq 2$.
We compute
\[
\xi_{k} \left( CT_{k}^{+}(s)y^s\right) = \frac{\bar s\Gamma(\overline{s}+k)}{\Gamma(\overline{s} + \frac{k}{2})}\htz(2 \overline{s}+k) y^{\overline{s} + k -1}.
\]
On the other hand,
\[
CT_{2-k}^{-}(- \overline{s})y^{\bar s+k-1} =
 (-1)^{1-\frac{k}{2}}\frac{\Gamma (1 - \overline{s} - \frac{k}{2})}{\Gamma (2-\overline{s} - k))}
  \htz(2 \overline{s} +k)
  y^{\overline{s}+k -1}.
\]
It remains to show that
\[
	\frac{\Gamma(s+k)}{\Gamma(s+\frac k2)} = (-1)^{1-\frac k2}(s+k-1) \frac{\Gamma(1-s-\frac k2)}{\Gamma(2-s-k)}.
\]
Indeed,
\[
	\frac{\Gamma(s+k)\Gamma(2-s-k)}{(s+k-1)\Gamma(s+\frac k2)\Gamma(1-s-\frac k2)} = -\frac{\Gamma(s+k)\Gamma(1-s-k)}{\Gamma(s+\frac k2)\Gamma(1-s-\frac k2)} = -\frac{\sin \pi(s+\frac k2)}{\sin\pi (s+k)} = (-1)^{1-\frac k2},
\]
as desired.
This completes the proof of \eqref{eq:CT-switch}.

We omit the details of similar assertions for $\xi_k (CT_{k}^{-}(s)y^{1-k-s})$
and for all other Fourier terms of index $n \ne 0$. 
For the cases  $n \ne 0$ various Whittaker function identities are required.
\end{proof}

\end{document}